\documentclass[amssymb,10pt]{article}
\usepackage[dvipdfm]{graphicx}
\usepackage{amsthm}
\usepackage{amsmath}
\usepackage{marvosym}
\usepackage{ifsym}
\usepackage[usenames]{color}
\usepackage{xspace,colortbl}
\usepackage{latexsym,url,amscd,amsmath,amssymb}
\usepackage{graphicx}
\usepackage{caption}
\usepackage{epstopdf}
\usepackage{subcaption}
\usepackage{multirow}
\usepackage{url}
\usepackage{framed}
\usepackage{algorithm}
\usepackage{appendix}
\usepackage{rotating}
\usepackage{float}
\usepackage[implicit=false]{hyperref}

\DeclareMathOperator*{\argmin}{arg\,min}

\DeclareMathOperator*{\dom}{dom}

\def\A{{\mathcal A}}
\def\B{{\mathcal B}}

\def\D{{\mathcal D}}

\def\F{{\mathcal F}}
\def\G{{\mathcal G}}
\def\H{{\mathcal H}}
\def\I{{\mathcal I}}

\def\L{{\mathcal L}}

\def\Q{{\mathcal Q}}
\def\R{{\mathcal R}}
\def\S{{\mathcal S}}
\def\T{{\mathcal T}}

\newtheorem{lemma}{Lemma}[section]

\newtheorem{theorem}{Theorem}[section]
\newtheorem{assumption}{Assumption}[section]

\allowdisplaybreaks[3]

\title{Convergence Analysis of Generalized ADMM with Majorization for Linearly Constrained Composite Convex Optimization}
\author{Hongwu Li\thanks{Schoole of Science, Beijing University of Technology, Beijing 100124, P.R. China (Email: lihongwu2018@163.com).}
\thanks{School of Mathematics and Statistics, Nanyang Normal University, Nanyang 473061, P.R. China.}
\and
Haibin Zhang\thanks{School of Science, Beijing University of Technology, Beijing 100124, P.R. China (Email: zhanghaibin@bjut.edu.cn).}
 \and
Yunhai Xiao\thanks{Center for Applied Mathematics of Henan Province,
	Henan University, Kaifeng 475000, P.R. China (Email: yhxiao@henu.edu.cn).} \href{https://orcid.org/0000-0002-7503-4585}{\includegraphics[scale=0.08]{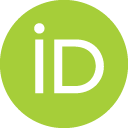}}
}

\begin{document}
\maketitle

\begin{abstract}
	The generalized alternating direction method of multipliers (ADMM) of Xiao et al. [{\tt Math. Prog. Comput., 2018}] aims at
	the two-block linearly constrained composite convex programming problem, in which each block is in the form of ``nonsmooth + quadratic".
	However, in the case of non-quadratic (but smooth), this method may fail unless the favorable structure of 'nonsmooth + smooth' is no longer used.
	This paper aims to remedy this defect by using a majorized technique to approximate the augmented Lagrangian function, so that the corresponding subprobllem can be decomposed into some smaller problems and then solved separately.
    Furthermore, the recent symmetric Gauss-Seidel (sGS) decomposition theorem guarantees the equivalence between the bigger subproblem and these smaller ones.
	This paper focuses on convergence analysis, that is,  we prove that the sequence generated by the proposed method converges globally to a Karush-Kuhn-Tucker point of the considered problem.
	Finally, we do some numerical experiments on a kind of simulated convex composite optimization problems which illustrate that the proposed method is more efficient than its compared ones.
\end{abstract}


\setcounter{equation}{0}
\section{Introduction}\label{section1}
Let $\mathbb{X}, \mathbb{Y}$, and $\mathbb{Z}$ be real finite dimensional Euclidean spaces with an inner product $\langle \cdot,\cdot \rangle $ and its induced norm $\|\cdot\|$.
Let  $f_1 : \mathbb{X} \rightarrow ( - \infty, + \infty)$ and $h_1 :  \mathbb{Y} \rightarrow ( - \infty, + \infty)$ are convex functions with Lipschitz continuous gradients; $f_2 : \mathbb{X} \rightarrow ( - \infty, + \infty]$ and $h_2 :  \mathbb{Y} \rightarrow ( - \infty, + \infty]$ are closed proper convex (not necessarily smooth) functions.
We consider the following composite convex optimization problem
\begin{equation}\label{prob1}
\min_{x\in\mathbb{X},y\in\mathbb{Y}} \ \big\{ f_1(x)+f_2(x) + h_1(y)+h_2(y) \ | \ \A^*x + \B^*y = c \big\},
\end{equation}
where  $\A : \mathbb{Z} \rightarrow \mathbb{X}$ and $\B : \mathbb{Z} \rightarrow \mathbb{Y}$ are linear operators with adjoints $\A^*$ and $\B^*$, respectively; and $c\in\mathbb{Z}$ is a given vector.

Let $\sigma\in( 0 , +\infty)$ be a penalty parameter, the augmented Lagrangian function associated with problem (\ref{prob1}) is defined by, for any $(x,y,z)\in\mathbb{X}\times\mathbb{Y}\times\mathbb{Z}$,
\begin{equation}\label{alf}
\L_{\sigma}(x,y;z) := f_1(x)+f_2(x) + h_1(y)+h_2(y) + \langle z, \A^*x + \B^*y - c\rangle +\frac{\sigma}{2}\|\A^*x + \B^*y - c\|^2,
\end{equation}
where $(x,y,z)\in\mathbb{X}\times\mathbb{Y}\times\mathbb{Z}$ and $z$ is a multiplier.
One attempt to solve (\ref{prob1}) is the standard alternating direction method of multipliers (ADMM), which alternately updates the variables $(x,y)$ and the multiplier $z$ from an initial point $(x^0,y^0,z^0)\in \dom(f)\times \dom(h)\times\mathbb{Z}$, that is,
\begin{equation}\label{cadmm}
\left\{
\begin{array}{l}
x^{k+1} = \argmin_{x\in\mathbb{X}}\big\{\L_{\sigma}(x,y^k;z^k)=:f_1(x)+f_2(x)+\frac{\sigma}{2}\|\A^*x + \B^*y^k - c+z^k/\sigma\|^2\big\},\\[3mm]
y^{k+1} =  \argmin_{y\in\mathbb{Y}}\big\{\L_{\sigma}(x^{k+1},y;z^k)=:h_1(y)+h_2(y)+\frac{\sigma}{2}\|\A^*x^{k+1} + \B^*y - c+z^k/\sigma\|^2\big\}, \\[3mm]
z^{k+1} = z^k + \tau\sigma(\A^*x^{k+1} + \B^*y^{k+1} - c),
\end{array}
\right.
\end{equation}
where $\tau\in(0, (1+\sqrt{5})/2)$ is a step-length. The convergence of the standard ADMM has long been established by Gabay \& Mercier\cite{GMC}, and Fortin \& Glowinski\cite{FGC}. And especially, Gabay \cite{ADMMDRS} showed that the standard ADMM with $\tau=1$ is exactly the Douglas-Rachford splitting method to the sum of two maximal monotone operators from the dual of (\ref{prob1}). And then, Eckstein \& Bertsekas \cite{GADMM} showed that itself is an instance of the proximal point algorithm \cite{ALMPPA} to a specially generated operator. For more work of ADMM, one refer to  an important note \cite{noteADMM} and excellent survey \cite{GHC}.

To improve the performance of (\ref{cadmm}) in the case of $\tau = 1$, Eckstein \& Bertsekas \cite{GADMM} also proposed a generalized variant of ADMM. Subsequently, Chen \cite[Section 3.2]{CHEND} made an interesting observation and concluded that the generalized ADMM of Eckstein \& Bertsekas \cite{GADMM} is equivalent to the following iterative scheme from an initial point $\widetilde{\omega}^0:=(\widetilde{x}^0, \widetilde{y}^0, \widetilde{z}^0)\in \dom(f)\times \dom(h)\times\mathbb{Z}$,
\begin{equation}\label{gadmm}
\left\{
\begin{array}{l}
x^{k} = \argmin_{x\in\mathbb{X}}\Big\{\L_{\sigma}(x,\widetilde{y}^k;\widetilde{z}^k)=:f_1(x)+f_2(x)+\frac{\sigma}{2}\|\A^*x + \B^*\widetilde{y}^k - c+\widetilde{z}^k/\sigma\|^2\Big\},\\[3mm]
z^{k} = \widetilde{z}^k + \sigma(\A^*x^{k} + \B^*\widetilde{y}^{k} - c),\\[3mm]
y^{k} =  \argmin_{y\in\mathbb{Y}}\Big\{\L_{\sigma}(x^{k},y;z^k)=:h_1(y)+h_2(y)+\frac{\sigma}{2}\|\A^*x^k + \B^*y - c+z^k/\sigma\|^2\Big\}, \\[3mm]
\widetilde{\omega}^{k+1} = \widetilde{\omega}^k + \rho(\omega^k - \widetilde{\omega}^k),
\end{array}
\right.
\end{equation}
where $\omega^k = (x^k, y^k, z^k)$, $\widetilde{\omega}^k = (\widetilde{x}^k, \widetilde{y}^k, \widetilde{z}^k)$, and $\rho\in(0,2)$ is an uniform relaxation factor. Clearly, for $\rho=1$, the above generalized ADMM scheme is scheme  (\ref{cadmm})  with $\tau=1$.
We see that the efficiency of the scheme (\ref{gadmm}) is mainly determined by the $x$- and $y$-subproblems. It was known that, if $f_1(x)$ and $h_1(y)$ are quadratic, and $f_2(x)$ and $g_2(y)$ are in the form of $f_2(x)=f_2(x_1)$ and $h_2(y)=h_2(y_1)$, the  $x$- and $y$-subproblems can be solved efficiently if a couple of suitable proximal point terms are added, that is
\begin{align}
		x^{k} = \argmin_{x\in\mathbb{X}}\Big\{f_1(x)+f_2(x_1)+\frac{\sigma}{2}\|\A^*x + \B^*\widetilde{y}^k - c+\widetilde{z}^k/\sigma\|^2+\frac12\|x-\widetilde{x}^k\|_\S^2\Big\},\label{xiaox}\\[2mm]
		y^{k} =  \argmin_{y\in\mathbb{Y}}\Big\{h_1(y)+h_2(y_1)+\frac{\sigma}{2}\|\A^*x^k + \B^*y - c+z^k/\sigma\|^2+\frac12\|y-\widetilde{y}^k\|_\T^2\Big\},\label{xiaoy}
\end{align}
where  $\S$ and $\T$ are two self-adjoint and positive semidefinite linear operators. It was shown from Li et al. \cite{spADMM,bsGs} that, if $\S$ and $\T$ are chosen properly, both subproblems can be split into some smaller ones and then solved separately in a sGS manner. However, if $f_1(x)$ and $h_1(y)$ are non-quadratic, it seemly difficult to split into some small pieces so that the favorable structure 'nonsmooth+smooth' cannot be used any more.
We must emphasize that the approach of (\ref{xiaox}) and (\ref{xiaoy}) was firstly proposed by Xiao et al. \cite{spGADMM} which has been numerically demonstrated very efficient to solve doubly non-negative semidefinite programming problems with moderate accuracy.

In optimization literature, one popular way to approximate a continuously differentially convex function into a quadratic is the using of majorization \cite{COVA}, such as, Hong et al. \cite{HONG} and Cui et al. \cite{CUIJOTA}.
Particularly, Li et al. \cite{miPADMM} majorized the augmented Lagrangian function involved in (\ref{cadmm}) when $f_1(x)$ and $g_1(y)$ are non-quadratic to make the subproblems become a composite convex quadratic minimization.
Therefore, the resulting problems were solved efficiently with respect to each variable in a sGS order \cite{bsGs}. The attractive feather of using sGS is to make this iterative scheme fill into the framework of Fazel et al. \cite{FST}.
In a similar way, Qin et al. \cite{MGADMM} also used a majorization technique to the generalized ADMM of Eckstein \& Bertsekas \cite{GADMM} to make it more flexible and robust. Extensive numerical experiments demonstrated that the generalized  ADMM with a suitable relaxation factor achieved better performance than the method of Li et al. \cite{miPADMM}.
Nevertheless, the performance of the semi-proximal generalized ADMM stated in (\ref{gadmm}), (\ref{xiaox}), and (\ref{xiaoy}) still has not been studied.
Hence, one natural question is how can we use the majorization technique so that the corresponding subproblems more amenable to efficient computations when $f_1(x)$ and $h_1(y)$ are non-quadratic.

The main purpose of this paper is to use a majorization technique to the augmented Lagrangian function (\ref{alf}) to make the resulting subproblems (\ref{xiaox}) and (\ref{xiaoy}) more easier solve, and hence, it will enhance the capabilities of the generalized ADMM of Xiao et al. \cite{spGADMM}.  At the beginning, we must clarify that the reason why we focus on the generalized ADMM of Xiao et al. \cite{spGADMM} is due to the fact that this type of method is highly efficient than some state-of-the-art algorithms according to a series of numerical experiments.
Let $\mathbb{X}:=\mathbb{X}_1\times\ldots\times\mathbb{X}_s$ and $\mathbb{Y}:=\mathbb{Y}_1\times\ldots\times\mathbb{Y}_t$ with positive constants $s$ and $t$. At each iteration, we use majorized functions to replace  $f_1(x)$ and $h_1(y)$ at the associated augmented Lagrangian function, and then both subproblems have ``nonsmooth+quadratic" structures. If $f_2(x):=f_2(x_1)$ and $h_2(y):=h_2(y_1)$ are simple functions, the $x$-subproblem ({\itshape resp.} $y$-subproblem) can be solved individually in the order of $x_s\rightarrow\ldots\rightarrow x_2\rightarrow x_1\rightarrow x_2\rightarrow\ldots\rightarrow x_s$ ({\itshape resp.} $y_t\rightarrow\ldots\rightarrow y_2\rightarrow y_1\rightarrow y_2\rightarrow\ldots\rightarrow y_t$) by making full use of the structures of $f_2(x_1)$ and $h_2(y_1)$.
Then, from the sGS decomposition theorem of Li et al. \cite{bsGs}, it is easy to show that this cycle is equivalent to adding proximal point terms with proper linear operators $\S$ and $\T$. We draw the difference in the iterative points between our method and the methods in \cite{MGADMM,miPADMM}, that is
$$
\begin{array}{rrcl}
[\text{Methods in \cite{miPADMM,MGADMM}	}] &\ldots\rightarrow(x^k,y^k)&\longrightarrow&(x^{k+1},y^{k+1})\rightarrow\ldots\\[2mm]
[\text{Our method}] &\ldots\rightarrow(x^k,y^k)&\rightarrow(\widetilde{x}^{k+1},\widetilde{y}^{k+1})\rightarrow&(x^{k+1},y^{k+1})\rightarrow\ldots
\end{array}
$$
which shows that the new point $(x^{k+1},y^{k+1})$ is from the relaxation point $(\widetilde{x}^{k+1},\widetilde{y}^{k+1})$ but not the previous $(x^k,y^k)$ as in \cite{miPADMM,MGADMM}.
We must emphasize that the relaxation point $(\widetilde{x}^{k},\widetilde{y}^{k})$ will lead to more technical difficulties, so that the theoretical analysis can not be obtained by mimicking the aforementioned methods in \cite{miPADMM,MGADMM}.
Most of the remainder of this paper will focus on theoretical analysis, that is, we prove that the sequence $\{x^k,y^k\}$ generated by our method converges to a Karush-Kuhn-Tucker (KKT) point of problem (\ref{prob1}) under some technical conditions. Finally, we do numerical experiments on a class of simulated convex composite optimization problems. The numerical results illustrate that the proposed method performs better than the M-ADMM of Li et al. \cite{miPADMM} and the M-GADMM of Qin et al. \cite{MGADMM}.

The remaining parts of this paper are organized as follows.  In section \ref{section3}, we propose the majorized and generalized ADMM to solve the composite convex problem (\ref{prob1}) in the case of $f_1(x)$ and $h_1(y)$ being non-quadratic and then it followed by some important properties.  Then, we focus on the convergence analysis of the proposed algorithm in section \ref{section4}. In section \ref{section5}, we are devoted to implementation issue to show the potential numerical efficiency of our proposed algorithm. Finally, we conclude this paper with some remarks in section \ref{section6}.

\setcounter{equation}{0}
\section{A generalized ADMM with majorization}\label{section3}
At the beginning of this section, we give some preliminaries needed in the subsequent developments.
Let $f:\mathbb{X}\rightarrow(-\infty,+\infty]$ be a closed proper convex function. The effective domain of $f$, which is denoted by $\dom(f)$, is defined as $\dom(f):=\{x:f(x)<+\infty\}$. A vector $x^*$ is said to be a subgradient of $f$ at point $x$ if $f(z)\geq f(x)+\langle x^*,z-x\rangle$ for all $z\in\mathbb{X}$. The set of all subgradients of $f$ at $x$ is called the subdifferential of $f$ at $x$ and is denoted by $\partial f(x)$ or $\partial f$.
It is well-known from \cite{COVA}
that $\partial f$ is a  maximal monotone operator. Because $f_1(x)$ and $h_1(y)$ are smooth convex functions with Lipschitz continuous gradients, we know that there exist self-adjoint and positive semidefinite linear operators such that $\widehat{\Sigma}_{f_1}\succeq \Sigma_{f_1}$ and $\widehat{\Sigma}_{h_1}\succeq \Sigma_{h_1}$, and for any $x, x' \in \mathbb{X}$ and any $y, y' \in \mathbb{Y}$, it holds that
\begin{equation}\label{f1cov}
	\frac{1}{2}\|x-x'\|_{\Sigma_{f_1}}^2 \le f_1(x) - f_1(x') - \langle x-x',\nabla f_1(x') \rangle \le \displaystyle{\frac{1}{2}}\|x-x'\|_{\widehat{\Sigma}_{f_1}}^2,
\end{equation}
\begin{equation}\label{h1cov}
	\frac{1}{2}\|y-y'\|_{\Sigma_{h_1}}^2 \le h_1(y) - h_1(y') - \langle y-y',\nabla h_1(y') \rangle \le \displaystyle{\frac{1}{2}}\|y-y'\|_{\widehat{\Sigma}_{h_1}}^2.
\end{equation}
Using the majorization technique, we construct the majorized functions for $f_1(x)$ and $h_1(y)$ as
\begin{equation}\label{hatf1}
	\hat{f}_1(x,x') := f_1(x') + \langle x-x',\nabla f_1(x') \rangle + \displaystyle{\frac{1}{2}}\|x-x'\|_{\hat{\Sigma}_{f_1}}^2,
\end{equation}
\begin{equation}\label{hath1}
	\hat{h}_1(y,y') := h_1(y') + \langle y-y',\nabla h_1(y') \rangle + \displaystyle{\frac{1}{2}}\|y-y'\|_{\hat{\Sigma}_{h_1}}^2.
\end{equation}
And then, using both majorized functions to replace the functions $f_1(x)$ and $h_1(y)$ in (\ref{alf}), we can get the majorized augmented Lagrangian function
\begin{equation}\label{majalf}
	\hat{\L}_{\sigma}\big(x,y;(z,x',y')\big) := \hat{f}_1(x,x') + f_2(x) + \hat{h}_1(y,y') + h_2(y) + \langle z, \A^*x + \B^*y - c\rangle +\frac{\sigma}{2}\|\A^*x + \B^*y - c\|^2.
\end{equation}
Clearly, it is a composite convex quadratic function except for the terms $f_2(x)$ and $h_2(y)$.

For subsequent developments, we need the following constraint qualification.
\begin{assumption}\label{assum}
	There exists $(x^0, y^0)\in ri(\dom(f_2) \times \dom (h_2)) \cap \Omega$, where $\Omega :=\{(x,y)\in \mathbb{X}\times\mathbb{Y}\mid \A^*x + \B^*y = c \}$.
\end{assumption}

Under Assumption (\ref{assum}), it is from \cite[Corollaries 28.2.2 and 28.3.1]{COVA}, we can get the optimality conditions of the problem (\ref{prob1}).
\begin{theorem}\label{optcond}
	If the Assumption (\ref{assum}) hold, then $(\bar{x}, \bar{y})\in \mathbb{X}\times\mathbb{Y}$ is an optimal solution to problem (\ref{prob1}) if and only if there
	exists a Lagrangian multiplier $\bar{z}\in\mathbb{Z}$ such that $(\bar{x},\bar{y},\bar{z})$ satisfies the following KKT system:
	\begin{equation}\label{kktcond}
		0\in \partial f_2(\bar{x}) + \nabla f_1(\bar{x}) + \A\bar{z},\quad 0\in \partial h_2(\bar{y}) + \nabla h_1(\bar{y}) + \B\bar{z},\quad \A^*\bar{x} + \B^*\bar{y} - c = 0,
	\end{equation}
	where $\bar z\in\mathbb{Z}$ is an optimal solution to the dual problem of (\ref{prob1}).
\end{theorem}
Because $f_2(x)$ and $h_2(y)$ are convex functions, the KKT system (\ref{kktcond}) is equivalent to finding a triple of points $(\bar{x},\bar{y},\bar{z})\in\mathbb{W} :=\mathbb{X}\times\mathbb{Y}\times\mathbb{Z}$ such that for any $(x, y)\in\mathbb{X}\times\mathbb{Y}$ the following inequality holds
\begin{equation}\label{covineq}
	\big(f_2(x) + h_2(y)\big) - \big(f_2(\bar{x}) + h_2(\bar{y})\big) + \big\langle x-\bar{x},\nabla f_1(\bar{x})+A\bar{z}\big\rangle + \big\langle y-\bar{y}, \nabla h_1(\bar{y})+B\bar{z}\big\rangle \ge 0.
\end{equation}
The inequality will be used frequently in the theoretical analysis part.

In light of above preliminary results, we are ready to the construct our algorithm.
For convenience, we denote
$$
\F:=\widehat{\Sigma}_{f_1} + \S + \sigma \A\A^*
\quad \text{and} \quad
\H:= \widehat{\Sigma}_{h_1} + \T + \sigma \B\B^*.
$$
At the current iteration, if we use the majorized augmented Lagrangian function (\ref{majalf}) to replace its standard form (\ref{alf}), the $x$-subproblem in (\ref{xiaox}) will reduce to
\begin{align*}
x^{k} =& \argmin_{x}\Big\{\hat{\L}_{\sigma}(x,\widetilde{y}^k;(\widetilde{z}^k,\widetilde{x}^k,\widetilde{y}^k)) + \frac{1}{2}\|x - \widetilde{x}^k\|_\S^2\Big\}\\[3mm]
=&\argmin_{x}\Big\{f_2(x) + \frac{1}{2}\langle x, \F x\rangle + \Big\langle \nabla f_1(\widetilde{x}^k) + \sigma \A(\A^*\widetilde{x}^{k} + \B^*\widetilde{y}^{k} - c + \sigma^{-1}\widetilde{z}^k) - \F \widetilde{x}^k, x\Big\rangle\Big\},
\end{align*}
and the $y$-subproblem in (\ref{xiaoy}) will reduce to
\begin{align*}
y^{k} =& \argmin_{y}\Big\{\hat{\L}_{\sigma}(x^k,y;(z^k,x^k,\widetilde{y}^k)) + \displaystyle{\frac{1}{2}}\|y - \widetilde{y}^k\|_\T^2\Big\}\\[3mm]
=&\argmin_{y}\Big\{h_2(y) + \displaystyle{\frac{1}{2}}\langle y, \H y\rangle + \Big\langle \nabla h_1(\widetilde{y}^k) + \sigma \B(\A^*x^{k} + \B^*\widetilde{y}^{k} - c + \sigma^{-1}z^{k}) - \H \widetilde{y}^k, y\Big\rangle\Big\}.
\end{align*}
Clearly, both subproblems have the favorable structure of ``nonsmooth + quadratic" so that they can be solved efficiently if the self-adjoint operators $\S$ and $\T$ are chosen properly. Taking the $x$-subproblem as an example, in the case of $f_2(x)=f_2(x_1)$ being a simple function, we denote $Q_x:=\widehat{\Sigma}_{f_1} + \sigma \A\A^*$ and then decompose it into $Q_x=U_x+\Sigma_x+U^\top_x$, where $U_x$ is a strictly upper triangular submatrix and $\Sigma_x$ is the diagonal of $Q_x$.  Let $\S:=U_x\Sigma_x^{-1}U_x^\top$, then the $x$-subproblem can be computed in a sGS order $x_s\rightarrow\ldots\rightarrow x_2\rightarrow x_1\rightarrow x_2\rightarrow\ldots\rightarrow x_s$, which indicates that the $x$-subproblem is split into a series of smaller problems with each $x_i$ and solved separately.
For more theoretical details on this iterative scheme, one may refer to the excellent papers of Li et al. \cite{spGADMM,miPADMM}.

In light of the above analyses, we list the  generalized ADMM with majorization (abbr. G-ADMM-M) for solving problem (\ref{prob1}) as follows.

\begin{framed}
\noindent
{\bf Algorithm: G-ADMM-M (Generalized ADMM with Majorization).}
\vskip 1.0mm \hrule \vskip 1mm
\noindent
\begin{itemize}
\item[Step 0.]  Choose $\sigma> 0$ and $\rho\in (0,2)$. Choose self-adjoint positive semidefinite linear operators $\S$ and $\T$  in $\mathbb{X}$ and $\mathbb{Y}$ such that $\F\succ 0$ and $\H\succ 0$.
 Input an initial point $\widetilde{\omega}^0 := (\widetilde{x}^0,\widetilde{y}^0,\widetilde{z}^0)\in \dom(f_2)\times \dom( h_2)\times \mathbb{Z}$. Let $k:=0$.

\item[Step 1.] Compute
\begin{equation}\label{mipgadmm}
\left\{
\begin{array}{l}
x^{k} = \argmin_{x}\Big\{f_2(x) + \displaystyle{\frac{1}{2}}\langle x, \F x\rangle + \langle \nabla f_1(\widetilde{x}^k) + \sigma \A(\A^*\widetilde{x}^{k} + \B^*\widetilde{y}^{k} - c + \sigma^{-1}\widetilde{z}^k) - \F \widetilde{x}^k, x\rangle\Big\},\\[3mm]
z^{k} = \widetilde{z}^k + \sigma(\A^*x^{k} + \B^*\widetilde{y}^{k} - c),\\[3mm]
y^{k} = \argmin_{y}\Big\{h_2(y) + \displaystyle{\frac{1}{2}}\langle y, \H y\rangle + \langle \nabla h_1(\widetilde{y}^k) + \sigma \B(\A^*x^{k} + \B^*\widetilde{y}^{k} - c + \sigma^{-1}z^{k}) - \H \widetilde{y}^k, y\rangle\Big\}.
\end{array}
\right.
\end{equation}

\item[Step 2.] Terminate if a termination criterion is satisfied. Output $(x^k,y^k,z^k)$. Otherwise, compute
\begin{equation}\label{omegak}
\widetilde{\omega}^{k+1} = \widetilde{\omega}^k + \rho(\omega^k - \widetilde{\omega}^k).
\end{equation}
Let $k:=k+1$ and go to Step 1.
\end{itemize}
\vspace{-.1cm}
\end{framed}

For the convergence analysis of the G-ADMM-M, we present a couple of useful lemmas. The results are well known in the field of numerical algebra, hence, we omit their proof here.
\begin{lemma}\label{salop1}
For any vectors $u$, $v$ in the same Euclidean vector space $\Im$, and any self-adjoint positive semidefinite linear operator $\G: \Im\rightarrow\Im$, it holds that
\begin{equation}\label{noneq}
\|u\|_\G^2 + \|v\|_\G^2 \ge \displaystyle{\frac{1}{2}}\|u - v\|_\G^2,
\end{equation}
and
\begin{equation}\label{ident1}
2\langle u, \G v\rangle = \|u\|_\G^2 + \|v\|_\G^2 - \|u - v\|_\G^2 = \|u + v\|_\G^2 - \|u\|_\G^2 - \|v\|_\G^2.
\end{equation}
\end{lemma}

\begin{lemma}\label{salop2}
	For any vectors $u_1$, $u_2$, $v_1$, and $v_2$ in the same Euclidean vector space $\Im$, and any self-adjoint positive semidefinite linear operator $\G: \Im\rightarrow\Im$, we have the identity
	\begin{equation}\label{ident2}
		2\langle u_1 - u_2, \G (v_1 - v_2)\rangle = \|u_1 - v_2\|_\G^2 + \|u_2 - v_1\|_\G^2 - \|u_1 - v_1\|_\G^2 - \|u_2 - v_2\|_\G^2.
	\end{equation}
\end{lemma}

Denote $\mathbb{W}^*$ be the set of solutions satisfy (\ref{kktcond}), which is nonempty under the Assumption \ref{assum}, i.e., the solution set of problem (\ref{prob1}) is nonempty. For $(\bar{x},\bar{y},\bar{z})\in\mathbb{W}^*$ and any $(x,y,z)\in\mathbb{W}$, we denote $x_e := x - \bar{x}$, $y_e := y - \bar{y}$, and $z_e := z -\bar{z}$ for convenience. Using these notations, we have the following three properties which will be used in our desiring convergence theorem analysis.

\begin{lemma}\label{lem3}
Suppose that Assumption \ref{assum} holds. Let $\{(x^k, y^k, z^k)\}$ be generated by Algorithm G-ADMM-M, and $(\bar{x},\bar{y},\bar{z})\in\mathbb{W}^*$. Then for any $\rho\in(0,2)$,  $\sigma>0$ and $k\ge 0$, we have
\begin{equation}\label{ident3}
\begin{array}{l}
\Big\langle \A^*x_e^{k+1} + \B^*y_e^k, z_e^{k+1} + \sigma(\rho - 1)\A^*x_e^{k+1}\Big\rangle \\[3mm]
=\displaystyle{\frac{1}{2\sigma\rho}}\Big(\|z_e^{k+1} + \sigma(\rho - 1)A^*x_e^{k+1}\|^2 - \|z_e^{k} + \sigma(\rho - 1)\A^*x_e^{k}\|^2\Big) + \displaystyle{\frac{\sigma\rho}{2}}\|\A^*x_e^{k+1} + \B^*y_e^k\|^2.
\end{array}
\end{equation}
\end{lemma}

\begin{proof}
From the iterative scheme(\ref{mipgadmm} and \ref{omegak}) in Algorithm G-ADMM-M, we have
\begin{equation}\label{lagmtran}
\begin{aligned}
z^{k+1} &= \widetilde{z}^{k+1} + \sigma(\A^*x^{k+1} + \B^*\widetilde{y}^{k+1} - c)\\[3mm]
&= z^k - (1-\rho)(z^k - \widetilde{z}^k) + \sigma\rho(\A^*x^{k+1} + \B^*y^k - c) + \sigma(1-\rho)(\A^*x^{k+1} + \B^*\widetilde{y}^k - c)\\[3mm]
&= z^k + \sigma\rho(\A^*x^{k+1} + \B^*y^k - c) + \sigma(\rho-1)(\A^*x^{k} - \A^*x^{k+1}),
\end{aligned}
\end{equation}
which indicates
$$
\big[z_e^{k+1} + \sigma(\rho - 1)\A^*x_e^{k+1}\big] - \big[z_e^{k} + \sigma(\rho - 1)\A^*x_e^{k}\big] = \sigma\rho(\A^*x_e^{k+1} + \B^*y_e^k).
$$
According to Lemma \ref{salop1}, we get
\begin{equation}\label{ident4}
\begin{array}{l}
2\sigma\rho\langle \A^*x_e^{k+1} + \B^*y_e^k, z_e^{k+1} + \sigma(\rho - 1)\A^*x_e^{k+1}\rangle \\[3mm]
=\sigma^2\rho^2\|\A^*x_e^{k+1} + \B^*y_e^k\|^2 + \|z_e^{k+1} + \sigma(\rho - 1)A^*x_e^{k+1}\|^2 - \|z_e^{k} + \sigma(\rho - 1)\A^*x_e^{k}\|^2.
\end{array}
\end{equation}
Combining with (\ref{ident4}), it yields the desired result (\ref{ident3}).
\end{proof}

\begin{lemma}\label{lem4}
Suppose that Assumption \ref{assum} holds. Let $\{(x^k, y^k, z^k)\}$ be generated by Algorithm G-ADMM-M, and $(\bar{x},\bar{y},\bar{z})\in\mathbb{W}^*$. Then for any $\rho\in(0,2)$, $\sigma>0$ and $k\ge 0$, we have
\begin{equation}\label{ident5}
\begin{array}{l}
\Big\langle \B^*y_e^k, z^{k} + \sigma(\A^*x_e^k + \B^*y_e^k) - z^{k+1} - \sigma(\rho - 1)\A^*x_e^{k+1}\Big\rangle \\[3mm]
=\displaystyle{\frac{\sigma(2-\rho)}{2}}\Big(\|\A^*x_e^{k} + \B^*y_e^k\|^2 - \|\A^*x_e^{k}\|^2\Big) - \displaystyle{\frac{\sigma\rho}{2}}\Big(\|\A^*x_e^{k+1} + B^*y_e^k\|^2 - \|\A^*x_e^{k+1}\|^2\Big).
\end{array}
\end{equation}
\end{lemma}
\begin{proof}
From (\ref{lagmtran}), we have
\begin{equation}\nonumber
\begin{array}{ll}
&z^{k} + \sigma(\A^*x_e^k + \B^*y_e^k) - z^{k+1} - \sigma(\rho - 1)\A^*x_e^{k+1} \\[3mm]
=&-\sigma\rho(\A^*x_e^{k+1} + \B^*y_e^k) + \sigma(\A^*x_e^{k} + \B^*y_e^k) - \sigma(\rho-1)\A^*x_e^{k}.
\end{array}
\end{equation}
Then we have
\begin{equation}\nonumber
\begin{array}{rl}
&\langle \B^*y_e^k, z^{k} + \sigma(\A^*x_e^k + \B^*y_e^k) - z^{k+1} - \sigma(\rho - 1)\A^*x_e^{k+1}\rangle \\[3mm]
=&-\sigma\rho\langle \B^*y_e^k, \A^*x_e^{k+1} + \B^*y_e^k\rangle + \sigma\langle \B^*y_e^k, \A^*x_e^{k} + \B^*y_e^k\rangle - \sigma(\rho-1)\langle \B^*y_e^k, \A^*x_e^{k}\rangle.
\end{array}
\end{equation}
According to the equality (\ref{ident1}) in Lemma \ref{salop1}, we have
\begin{equation}\nonumber
\begin{array}{rl}
&-\sigma\rho\langle \B^* y_e^k, \A^*x_e^{k+1} + \B^*y_e^k\rangle + \sigma\langle \B^* y_e^k, \A^*x_e^{k} + \B^*y_e^k\rangle - \sigma(\rho-1)\langle \B y_e^k, \A^*x_e^{k}\rangle\\[3mm]
= &- \frac{\sigma\rho}{2}\Big(\|\A^*x_e^{k+1} + \B^*y_e^k\|^2 + \|B^*y_e^k\|^2 - \|\A^*x_e^{k+1}\|^2\Big) + \displaystyle{\frac{\sigma}{2}}\Big(\|\A^*x_e^{k} + \B^*y_e^k\|^2 + \|\B^*y_e^k\|^2 - \|\A^*x_e^{k}\|^2\Big) \\[3mm]
& - \frac{\sigma(\rho- 1)}{2}\Big(\|\A^*x_e^{k} + \B^*y_e^k\|^2 - \|\B^*y_e^k\|^2 - \|\A^*x_e^{k}\|^2\Big)\\[3mm]
= &\frac{\sigma(2-\rho)}{2}\Big(\|A^*x_e^{k} + B^*y_e^k\|^2 - \|A^*x_e^{k}\|^2\Big) - \displaystyle{\frac{\sigma\rho}{2}}\Big(\|A^*x_e^{k+1} + B^*y_e^k\|^2 - \|A^*x_e^{k+1}\|^2\Big).
\end{array}
\end{equation}
The proof is complete.
\end{proof}

For notational convenience, we define
\begin{equation}\label{psik}
	\Psi_k(\bar{x},\bar{y},\bar{z}) := \frac{1}{\sigma\rho}\|z_e^k + \sigma(\rho-1)A^*x_e^{k}\|^2 + \sigma(2-\rho)\|A^*x_e^{k}\|^2 + \frac{1}{\rho}\|\widetilde{x}_e^{k+1}\|_{\widehat{\Sigma}_{f_1}+\S}^2 + \frac{1}{\rho}\|\widetilde{y}_e^k\|_{\widehat{\Sigma}_{h_1}+\T}^2,
\end{equation}
and
\begin{equation}\label{deltak}
	\begin{array}{l}
		\delta_k := \|\widetilde{x}^{k+1} - x^{k+1}\|_{\frac{1}{2}\Sigma_{f_1}}^2 + \|\widetilde{y}^{k} - y^{k}\|_{\frac{1}{2}\Sigma_{h_1}+(2-\rho)(\widehat{\Sigma}_{h_1}+\T)}^2 + (1-\lambda)(2-\rho)\|\widetilde{x}^{k} - x^{k}\|_{\widehat{\Sigma}_{f_1}+\S}^2\\[3mm]
		\qquad + \sigma(2\lambda-1)(2-\rho)\|\A^*x^{k+1} + \B^*y^{k} - c\|^2 + \frac{\sigma(1-\lambda)(2-\rho)}{2}\|\B^*(y^k-y^{k-1})\|^2\\[3mm]
		\qquad + \frac{\lambda(2-\rho)^2}{\rho}\|x^{k+1} - x^k\|_{\widehat{\Sigma}_{f_1}+\S+\sigma \A\A^*}^2,
	\end{array}
\end{equation}
where $\lambda>0$ will defined later.
Furthermore, we also define
\begin{equation}\label{thetak}
	\theta_k := \|\widetilde{x}^{k+1} - x^{k+1}\|_{\frac{1}{2}\Sigma_{f_1}-\hat{\Sigma}_{f_1}+(2-\rho)(\widehat{\Sigma}_{f_1}+\S)}^2 + \|\widetilde{y}^{k} - y^{k}\|_{\frac{1}{2}\Sigma_{h_1}-\hat{\Sigma}_{h_1}+(2-\rho)(\widehat{\Sigma}_{h_1}+\T)}^2,
\end{equation}

\begin{equation}\label{etak}
	\eta_k := \Big\langle \A^*x_e^{k+1}, z_e^{k+1} \Big\rangle + \Big\langle \B^*y_e^k, z_e^k + \sigma(\A^*x^k + \B^*y^k - c)\Big\rangle,
\end{equation}
and
\begin{equation}\label{xik}
	\xi_k := \|\widetilde{x}^{k+1} - x^{k+1}\|_{\widehat{\Sigma}_{f_1}}^2 + \|\widetilde{y}^{k} - y^{k}\|_{\widehat{\Sigma}_{h_1}}^2.
\end{equation}

For the global convergence for the Algorithm G-ADMM-M, the following inequality is essential.

\begin{lemma}\label{lem5}
Suppose that Assumption \ref{assum} holds. Let $\{(x^k, y^k, z^k)\}$ be generated by Algorithm G-ADMM-M, and $(\bar{x},\bar{y},\bar{z})\in\mathbb{W}^*$. Let $\Psi_k$ and $\theta_k$ be defined as in (\ref{psik}) and (\ref{thetak}). Then for any $\rho\in(0,2)$, $\sigma>0$ and $k\ge 0$, we have
\begin{equation}\label{pivineq}
\Psi_k(\bar{x},\bar{y},\bar{z}) - \Psi_{k+1}(\bar{x},\bar{y},\bar{z}) \ge \theta_k + \sigma(2-\rho)\|\A^*x^k + \B^*y^k - c\|^2.
\end{equation}
\end{lemma}
\begin{proof}
	From the left inequality of (\ref{f1cov}), let $x':=\widetilde{x}^{k+1}$, we have
	\begin{equation}\label{ineqcovl}
		f_1(x) \ge f_1(\widetilde{x}^{k+1}) + \langle x-\widetilde{x}^{k+1},\nabla f_1(\widetilde{x}^{k+1}) \rangle + \displaystyle{\frac{1}{2}}\|x-\widetilde{x}^{k+1}\|_{\Sigma_{f_1}}^2,
	\end{equation}
	and then from the right inequality of (\ref{f1cov}), let $x:=x^{k+1}$, $x':=\widetilde{x}^{k+1}$, we have
	\begin{equation}\label{ineqcovr}
		f_1(x^{k+1}) \le f_1(\widetilde{x}^{k+1}) + \langle x^{k+1}-\widetilde{x}^{k+1},\nabla f_1(\widetilde{x}^{k+1}) \rangle + \displaystyle{\frac{1}{2}}\|x^{k+1}-\widetilde{x}^{k+1}\|_{\widehat{\Sigma}_{f_1}}^2.
	\end{equation}
Subtracting both sides of (\ref{ineqcovl}) and (\ref{ineqcovr}), we obtain
\begin{equation}\label{ineq1}
	f_1(x) - f_1(x^{k+1}) - \langle x-x^{k+1},\nabla f_1(\widetilde{x}^{k+1}) \rangle \ge \displaystyle{\frac{1}{2}}\|x - \widetilde{x}^{k+1}\|_{\Sigma_{f_1}}^2 - \displaystyle{\frac{1}{2}}\|\widetilde{x}^{k+1} - x^{k+1}\|_{\widehat{\Sigma}_{f_1}}^2.
\end{equation}
Using the necessary optimality conditions of $x^{k+1}$ in (\ref{mipgadmm}), we have for any $x\in\mathbb{X}$ and $\xi\in\partial f_2(x^{k+1})$ that
\begin{equation}\label{varineq}
	\Big\langle x - x^{k+1}, \xi + \F x^{k+1} + \nabla f_1(\widetilde{x}^{k+1}) + \sigma \A(\A^*\widetilde{x}^{k+1} + \B^*\widetilde{y}^{k+1} - c + \sigma^{-1}\widetilde{z}^{k+1}) - \F\widetilde{x}^{k+1} \Big\rangle \ge 0.
\end{equation}
From the convexity of $f_2(x)$, we get
\begin{equation}\label{convf2}
	\langle x - x^{k+1}, \xi\rangle \le f_2(x) - f_2(x^{k+1}).
\end{equation}
Then, substituting (\ref{convf2}) into (\ref{varineq}) and using the relaxation step (\ref{omegak}), it yields
\begin{align}\label{ineq2}
		&f_2(x) - f_2(x^{k+1}) + \Big\langle x - x^{k+1}, \nabla f_1(\widetilde{x}^{k+1}) + \sigma \A(\A^*\widetilde{x}^{k+1} + \B^*\widetilde{y}^{k+1} - c + \sigma^{-1}\widetilde{z}^{k+1}) + \F(x^{k+1} - \widetilde{x}^{k+1})\Big\rangle\nonumber\\[3mm]
		=& f_2(x) - f_2(x^{k+1}) + \Big\langle x - x^{k+1}, \nabla f_1(\widetilde{x}^{k+1}) + \A z^{k+1} - (\widehat{\Sigma}_{f_1}+\S)(\widetilde{x}^{k+1} - x^{k+1})\Big\rangle \ge 0.
\end{align}
Combining (\ref{ineq1}) and (\ref{ineq2}), we obtain
\begin{equation}\label{ineq3}
	\begin{array}{l}
		f_1(x) - f_1(x^{k+1}) + f_2(x) - f_2(x^{k+1}) + \langle x - x^{k+1}, \A z^{k+1} - (\widehat{\Sigma}_{f_1}+\S)(\widetilde{x}^{k+1} - x^{k+1})\rangle\\[3mm]
		\ge \frac{1}{2}\|x - \widetilde{x}^{k+1}\|_{\Sigma_{f_1}}^2 - \frac{1}{2}\|\widetilde{x}^{k+1} - x^{k+1}\|_{\widehat{\Sigma}_{f_1}}^2.
	\end{array}
\end{equation}
Similarly, for any $y\in\mathbb{Y}$, we have
\begin{equation}\label{ineq4}
	\begin{array}{l}
		h_1(y) - h_1(y^{k}) + h_2(y) - h_2(y^{k}) + \Big\langle y - y^{k}, \sigma \B(\A^*x^{k} + \B^*y^{k} - c + \sigma^{-1}z^{k}) - (\hat{\Sigma}_{h_1}+\T)(\widetilde{y}^{k} - y^{k})\Big\rangle\\[3mm]
		\ge \frac{1}{2}\|y - \widetilde{y}^{k}\|_{\Sigma_{h_1}}^2 - \frac{1}{2}\|\widetilde{y}^{k} - y^{k}\|_{\widehat{\Sigma}_{h_1}}^2.
	\end{array}
\end{equation}
Let $x=\bar{x}$ and $y=\bar{y}$, and combine (\ref{ineq3}) with (\ref{ineq4}), we obtain
\begin{equation}\label{ineq5}
	\begin{array}{l}
		f_1(\bar{x}) - f_1(x^{k+1}) + h_1(\bar{y}) - h_1(y^{k}) + f_2(\bar{x}) - f_2(x^{k+1}) + h_2(\bar{y}) - h_2(y^{k})\\[3mm]
		 - \Big\langle x_e^{k+1}, \A z^{k+1} - (\hat{\Sigma}_{f_1}+\S)(\widetilde{x}^{k+1} - x^{k+1})\Big\rangle - \Big\langle y_e^{k}, \sigma \B(\A^*x^{k} + \B^*y^{k} - c + \sigma^{-1}z^{k}) - (\widehat{\Sigma}_{h_1}+\T)(\widetilde{y}^{k} - y^{k})\Big\rangle\\[3mm]
		\ge \displaystyle{\frac{1}{2}}\|\widetilde{x}_e^{k+1}\|_{\Sigma_{f_1}}^2 - \displaystyle{\frac{1}{2}}\|\widetilde{x}^{k+1} - x^{k+1}\|_{\widehat{\Sigma}_{f_1}}^2 + \displaystyle{\frac{1}{2}}\|\widetilde{y}_e^{k}\|_{\Sigma_{h_1}}^2 - \displaystyle{\frac{1}{2}}\|\widetilde{y}^{k} - y^{k}\|_{\widehat{\Sigma}_{h_1}}^2.
	\end{array}
\end{equation}
From (\ref{f1cov}) and (\ref{h1cov}), we can easily get
\begin{equation}\label{ineq6}
	f_1(x^{k+1}) - f_1(\bar{x}) \ge \langle x_e^{k+1},\nabla f_1(\bar{x}) \rangle + \displaystyle{\frac{1}{2}}\|x_e^{k+1}\|_{\Sigma_{f_1}}^2,
\end{equation}
and
\begin{equation}\label{ineq7}
	h_1(y^{k}) - h_1(\bar{y}) \ge \langle y_e^{k},\nabla h_1(\bar{y}) \rangle + \displaystyle{\frac{1}{2}}\|y_e^{k}\|_{\Sigma_{h_1}}^2.
\end{equation}
According to (\ref{ineq5}), (\ref{ineq6}), (\ref{ineq7}) and the inequality (\ref{noneq}), we obtain
\begin{align}\label{ineq8}
		&f_2(\bar{x}) - f_2(x^{k+1}) + h_2(\bar{y}) - h_2(y^{k}) - \langle x_e^{k+1}, \nabla f_1(\bar{x}) + A\bar{z}\rangle - \langle y_e^{k}, \nabla h_1(\bar{y}) + \B\bar{z}\rangle - \langle \A^*x_e^{k+1}, z_e^{k+1} \rangle\nonumber\\[3mm]
		&- \Big\langle \B^*y_e^k, z_e^k + \sigma(\A^*x^k + \B^*y^k - c)\Big\rangle + \Big\langle x_e^{k+1}, (\hat{\Sigma}_{f_1}+\S)(\widetilde{x}^{k+1} - x^{k+1})\Big\rangle + \Big\langle y_e^{k}, (\hat{\Sigma}_{h_1}+\T)(\widetilde{y}^{k} - y^{k})\Big\rangle\nonumber\\[3mm]
		\ge& \frac{1}{2}\|\widetilde{x}_e^{k+1}\|_{\Sigma_{f_1}}^2 - \frac{1}{2}\|\widetilde{x}^{k+1} - x^{k+1}\|_{\hat{\Sigma}_{f_1}}^2 + \displaystyle{\frac{1}{2}}\|\widetilde{y}_e^{k}\|_{\Sigma_{h_1}}^2 - \frac{1}{2}\|\widetilde{y}^{k} - y^{k}\|_{\hat{\Sigma}_{h_1}}^2 + \frac{1}{2}\|x_e^{k+1}\|_{\Sigma_{f_1}}^2 + \displaystyle{\frac{1}{2}}\|y_e^{k}\|_{\Sigma_{h_1}}^2\nonumber\\[3mm]
		\ge& \displaystyle{\frac{1}{4}}\|\widetilde{x}^{k+1} - x^{k+1}\|_{\Sigma_{f_1}}^2 + \displaystyle{\frac{1}{4}}\|\widetilde{y}^{k} - y^{k}\|_{\Sigma_{h_1}}^2 - \displaystyle{\frac{1}{2}}\|\widetilde{x}^{k+1} - x^{k+1}\|_{\widehat{\Sigma}_{f_1}}^2 - \displaystyle{\frac{1}{2}}\|\widetilde{y}^{k} - y^{k}\|_{\widehat{\Sigma}_{h_1}}^2.
\end{align}
Replace $x$ with $x^{k+1}$, and $y$ with $y^k$ in the inequality (\ref{covineq}), we get
$$
f_2(\bar{x}) - f_2(x^{k+1}) + h_2(\bar{y}) - h_2(y^{k}) - \langle x_e^{k+1}, \nabla f_1(\bar{x}) + A\bar{z}\rangle - \langle y_e^{k}, \nabla h_1(\bar{y}) + B\bar{z}\rangle \le 0.
$$
Substituting this inequality and (\ref{etak}) into (\ref{ineq8}), it yields
\begin{align}\label{ineq9}
		&-\eta^k + \Big\langle x_e^{k+1}, (\hat{\Sigma}_{f_1}+\S)(\widetilde{x}^{k+1} - x^{k+1})\Big\rangle + \Big\langle y_e^{k}, (\hat{\Sigma}_{h_1}+\T)(\widetilde{y}^{k} - y^{k})\Big\rangle\nonumber\\[3mm]
		\ge& \frac{1}{4}\|\widetilde{x}^{k+1} - x^{k+1}\|_{\Sigma_{f_1}}^2 + \frac{1}{4}\|\widetilde{y}^{k} - y^{k}\|_{\Sigma_{h_1}}^2 - \frac{1}{2}\|\widetilde{x}^{k+1} - x^{k+1}\|_{\widehat{\Sigma}_{f_1}}^2 - \frac{1}{2}\|\widetilde{y}^{k} - y^{k}\|_{\widehat{\Sigma}_{h_1}}^2.
\end{align}
Using Lemmas \ref{lem3} and \ref{lem4}, we have
\begin{equation}\label{etakconv}
	\begin{array}{l}
		\eta^k = \Big\langle \A^*x_e^{k+1}, z_e^{k+1} \Big\rangle + \Big\langle \B^*y_e^k, z_e^k + \sigma(\A^*x^k + \B^*y^k - c)\Big\rangle\\[3mm]
		\quad = \Big\langle \A^*x_e^{k+1} + \B^*y_e^k, z_e^{k+1} + \sigma(\rho - 1)\A^*x_e^{k+1}\Big\rangle - \Big\langle \B^*y_e^k, z_e^{k+1} + \sigma(\rho - 1)A^*x_e^{k+1}\Big\rangle\\[3mm]
		\qquad - \Big\langle \A^*x_e^{k+1}, \sigma(\rho - 1)\A^*x_e^{k+1}\Big\rangle + \Big\langle \B^*y_e^k, z_e^k + \sigma(\A^*x^k + \B^*y^k - c)\Big\rangle\\[3mm]
		\quad = \Big\langle \A^*x_e^{k+1} + \B^*y_e^k, z_e^{k+1} + \sigma(\rho - 1)\A^*x_e^{k+1}\Big\rangle + \Big\langle \B^*y_e^k, z_e^k + \sigma(\A^*x^k + \B^*y^k - c) - z_e^{k+1}\\[3mm]
		\qquad - \sigma(\rho - 1)\A^*x_e^{k+1}\Big\rangle - \sigma(\rho - 1)\|\A^*x_e^{k+1}\|^2\\[3mm]
		\quad = \displaystyle{\frac{1}{2\sigma\rho}}\Big(\|z_e^{k+1} + \sigma(\rho - 1)\A^*x_e^{k+1}\|^2 - \|z_e^{k} + \sigma(\rho - 1)\A^*x_e^{k}\|^2\Big)\\[3mm] 	\qquad+ \displaystyle{\frac{\sigma(2-\rho)}{2}}\Big(\|\A^*x^{k} + \B^*y^k - c\|^2
	 + \|\A^*x_e^{k+1}\|^2 - \|\A^*x_e^{k}\|^2\Big).\\[3mm]
	\end{array}
\end{equation}
From (\ref{omegak}) and (\ref{ident2}), we have
\begin{equation}\label{xekconv}
	\begin{array}{l}
		\Big\langle x_e^{k+1}, (\widehat{\Sigma}_{f_1}+\S)(\widetilde{x}^{k+1} - x^{k+1})\Big\rangle\\[3mm]
		= \displaystyle{\frac{1}{\rho}}\Big\langle x^{k+1} - \bar{x}, (\widehat{\Sigma}_{f_1}+\S)(\widetilde{x}^{k+1} - \widetilde{x}^{k+2})\Big\rangle\\[3mm]
		= \displaystyle{\frac{1}{2\rho}}\Big(\|\widetilde{x}_e^{k+1}\|_{\widehat{\Sigma}_{f_1}+\S}^2 - \|\widetilde{x}_e^{k+2}\|_{\widehat{\Sigma}_{f_1}+\S}^2 - \rho(2-\rho)\|x^{k+1} - \widetilde{x}^{k+1}\|_{\widehat{\Sigma}_{f_1}+\S}^2\Big).
	\end{array}
\end{equation}
In a similar way, we get
\begin{equation}\label{yekconv}
	\begin{array}{l}
		\Big\langle y_e^{k}, (\hat{\Sigma}_{h_1}+\T)(\widetilde{y}^{k} - y^{k})\Big\rangle\\[3mm]
		= \displaystyle{\frac{1}{2\rho}}\Big(\|\widetilde{y}_e^{k}\|_{\widehat{\Sigma}_{h_1}+\T}^2 - \|\widetilde{y}_e^{k+1}\|_{\widehat{\Sigma}_{h_1}+\T}^2 - \rho(2-\rho)\|y^{k} - \widetilde{y}^{k}\|_{\widehat{\Sigma}_{h_1}+\T}^2\Big).
	\end{array}
\end{equation}
Combining (\ref{ineq9}), (\ref{etakconv}), (\ref{xekconv}) with (\ref{yekconv}), we obtain
\begin{align}\label{ineqconv}
&\frac{1}{2\sigma\rho}\Big(\|z_e^{k} + \sigma(\rho - 1)\A^*x_e^{k}\|^2 - \|z_e^{k+1} + \sigma(\rho - 1)\A^*x_e^{k+1}\|^2\Big) - \frac{\sigma(2-\rho)}{2}\|\A^*x^{k} + \B^*y^k - c\|^2\nonumber\\[3mm]
		&+ \frac{\sigma(2-\rho)}{2}\Big(\|A^*x_e^{k}\|^2 - \|A^*x_e^{k+1}\|^2\Big) - \frac{2-\rho}{2}\|x^{k+1} - \widetilde{x}^{k+1}\|_{\widehat{\Sigma}_{f_1}+\S}^2 \nonumber\\[3mm]
		&+ \frac{1}{2\rho}\Big(\|\widetilde{x}_e^{k+1}\|_{\hat{\Sigma}_{f_1}+\S}^2- \|\widetilde{x}_e^{k+2}\|_{\widehat{\Sigma}_{f_1}+\S}^2\Big) - \frac{2-\rho}{2}\|y^{k} - \widetilde{y}^{k}\|_{\widehat{\Sigma}_{h_1}+\T}^2 + \frac{1}{2\rho}\Big(\|\widetilde{y}_e^{k}\|_{\hat{\Sigma}_{h_1}+\T}^2 - \|\widetilde{y}_e^{k+1}\|_{\widehat{\Sigma}_{h_1}+\T}^2\Big)\nonumber\\[3mm]
		\ge &
		\frac{1}{4}\|\widetilde{x}^{k+1} - x^{k+1}\|_{\Sigma_{f_1}}^2 + \frac{1}{4}\|\widetilde{y}^{k} - y^{k}\|_{\Sigma_{h_1}}^2 - \frac{1}{2}\|\widetilde{x}^{k+1} - x^{k+1}\|_{\widehat{\Sigma}_{f_1}}^2 - \frac{1}{2}\|\widetilde{y}^{k} - y^{k}\|_{\widehat{\Sigma}_{h_1}}^2.
\end{align}
Substitutiing (\ref{psik}) and (\ref{thetak}) into  (\ref{ineqconv}) and noting the definition of $\Psi_k$ in (\ref{psik}), we can get the conclusion (\ref{pivineq}) immediately.

\end{proof}

\newpage

\setcounter{equation}{0}
\section{Convergence analysis}\label{section4}
Based on the lemmas proved in the previous section,  we can establish the convergence of G-ADMM-M for solving problem (\ref{prob1}). For this purpose, we firstly give some useful lemmas.
\begin{lemma}\label{lem6}
Suppose that the Assumption (\ref{assum}) hold. Let the sequence $\{(x^k,y^k,z^k)\}$ be generated by Algorithm G-ADMM-M and $(\bar{x},\bar{y},\bar{z})\in \mathbb{W}^*$. Then, for any $\sigma > 0$, $\rho\in(0,2)$, and $k > 1$, we have
\begin{align}\label{conneq}
\sigma(2-\rho)\|\A^*x^{k} + \B^*y^k - c\|^2 \ge& \sigma(2-\rho)\|\A^*x^{k+1} + \B^*y^k - c\|^2 + \frac{(2-\rho)^2}{\rho}\|x^{k+1} - x^k\|_{\widehat{\Sigma}_{f_1}+\S+\sigma AA^*}\nonumber\\[3mm]
 &- (2-\rho)\|\widetilde{x}^k - x^k\|_{\widehat{\Sigma}_{f_1}+\S}^2.
 \end{align}
\end{lemma}
\begin{proof}
Note that
\begin{align*}
&\sigma(2-\rho)\|\A^*x^{k} + \B^*y^k - c\|^2\\[3mm]
= &\sigma(2-\rho)\|A^*x^{k+1} + \B^*y^k - c + A^*(x^k - x^{k+1})\|^2\\[3mm]
= &\sigma(2-\rho)\|A^*x^{k+1} + \B^*y^k - c\|^2 + \sigma(2-\rho)\|A^*(x^k - x^{k+1})\|^2\\[3mm]
 &+ 2\sigma(2-\rho)\Big\langle \A^*x^{k+1} + \B^*y^k - c, \A^*(x^k - x^{k+1})\Big\rangle.
 \end{align*}
From (\ref{lagmtran}), we have
\begin{align*}
&2\sigma(2-\rho)\Big\langle \A^*x^{k+1} + \B^*y^k - c, A^*(x^k - x^{k+1})\Big\rangle\\[3mm]
=& \frac{2(2-\rho)}{\rho}\Big\langle z_e^{k+1} - z_e^k + \sigma(\rho-1)(\A^*x_e^{k+1} - \A^*x_e^k), \A^*(x^k - x^{k+1})\Big\rangle\\[3mm]
= &\frac{2(2-\rho)}{\rho}\Big\langle z_e^{k+1} - z_e^k, \A^*(x^k - x^{k+1})\Big\rangle - \frac{2\sigma(\rho-1)(2-\rho)}{\rho}\|\A^*(x^k - x^{k+1})\|^2.
\end{align*}
Then we have
\begin{align}\label{eqcon}
&\sigma(2-\rho)\|\A^*x^{k} + \B^*y^k - c\|^2\nonumber\\[3mm]
 = &\sigma(2-\rho)\|\A^*x^{k+1} + \B^*y^k - c\|^2 +\frac{(2-\rho)^2}{\rho}\|x^{k+1} - x^{k}\|_{\sigma \A\A^*}^2
+ \frac{2(2-\rho)}{\rho}\Big\langle z_e^{k+1} - z_e^k, \A^*(x^k - x^{k+1})\Big\rangle.
\end{align}
According to the first-order optimality condition of the x-subproblem in (\ref{mipgadmm}), we know that
\begin{equation}\label{foocx}
\begin{array}{l}
-\nabla f_1(\widetilde{x}^k) - (\widehat{\Sigma}_{f_1} + \S)(x^k - \widetilde{x}^k) - \A z^k \in \partial f_2(x^k),\\[3mm]
-\nabla f_1(\widetilde{x}^{k+1}) - (\widehat{\Sigma}_{f_1} + \S)(x^{k+1} - \widetilde{x}^{k+1}) - \A z^{k+1} \in \partial f_2(x^{k+1}).
 \end{array}
\end{equation}
Hence, from the convexity of $f_2(x)$, we have
\begin{equation}\nonumber
\Big\langle x^k - x^{k+1}, \nabla f_1(\widetilde{x}^{k+1}) - \nabla f_1(\widetilde{x}^k) + (\widehat{\Sigma}_{f_1} + \S)[(x^{k+1} - x^k) - (\widetilde{x}^{k+1} - \widetilde{x}^k)] + \A(z^{k+1} - z^k)\Big\rangle\ge 0.
\end{equation}
Then, it yields that
\begin{align}\label{neqsubxz}
&\langle z^{k+1} - z^k, \A^*(x^k - x^{k+1})\rangle\nonumber\\[3mm]
=&\langle \A(z^{k+1} - z^k), x^k - x^{k+1}\rangle\nonumber\\[3mm]
\ge& \Big\langle x^{k+1} - x^k, \nabla f_1(\widetilde{x}^{k+1}) - \nabla f_1(\widetilde{x}^k)\Big\rangle - \Big\langle x^{k+1} - x^k, (\widehat{\Sigma}_{f_1} + \S)(\widetilde{x}^{k+1} - \widetilde{x}^k)\Big\rangle\nonumber\\[3mm]
 &+ \|x^{k+1} - x^k\|_{\widehat{\Sigma}_{f_1} + \S}^2.
\end{align}
Noting that $\nabla f_1$ is assumed be
globally Lipschitz continuous, then, from Clarkes Mean-Value Theorem \cite[Proposition 2.6.5]{OPNSA}, there exists a self-adjoint and positive semidefinite linear operator $\Gamma^k\in\partial^2(f_1(\gamma^k))$ with $\gamma^k$ being a point at the segment between $x^k$ and $x^{k+1}$, such that
\begin{equation}\nonumber
\nabla f_1(\widetilde{x}^{k+1}) - \nabla f_1(\widetilde{x}^k) = \Gamma^k(\widetilde{x}^{k+1} - \widetilde{x}^{k}).
\end{equation}
Then, according to (\ref{omegak}) and the left equality in (\ref{ident1}), the inequality (\ref{neqsubxz}) can be reorganized as
\begin{equation}\label{neqsubxzd}
\begin{array}{l}
\Big\langle z^{k+1} - z^k, \A^*(x^k - x^{k+1})\Big\rangle\\[3mm]
\ge \Big\langle x^{k+1} - x^k, (\widehat{\Sigma}_{f_1} + \S - \Gamma^k)(\widetilde{x}^{k} - \widetilde{x}^{k+1})\Big\rangle + \|x^{k+1} - x^k\|_{\widehat{\Sigma}_{f_1} + \S}^2\\[3mm]
 = \rho \Big\langle x^{k+1} - x^k, (\widehat{\Sigma}_{f_1} + \S - \Gamma^k)(\widetilde{x}^{k} - x^{k})\Big\rangle + \|x^{k+1} - x^k\|_{\widehat{\Sigma}_{f_1} + \S}^2\\[3mm]
 = \displaystyle{\frac{\rho}{2}}\Big[\|x^{k+1} - x^k\|_{\widehat{\Sigma}_{f_1} + \S - \Gamma^k}^2 + \|\widetilde{x}^k - x^k\|_{\widehat{\Sigma}_{f_1} + \S - \Gamma^k}^2 - \|\widetilde{x}^k - x^{k+1}\|_{\widehat{\Sigma}_{f_1} + \S - \Gamma^k}^2\Big] + \|x^{k+1} - x^k\|_{\widehat{\Sigma}_{f_1} + \S}^2.
 \end{array}
\end{equation}
Because $\widehat{\Sigma}_{f_1}\succeq \Gamma^k\succeq 0$ and $\frac{1}{2}\widehat{\Sigma}_{f_1} + \S\succeq 0$, we have
$$
\widehat{\Sigma}_{f_1} + \S - \frac{1}{2}\Gamma^k =  \frac{1}{2}\widehat{\Sigma}_{f_1} + \S + \frac{1}{2}(\widehat{\Sigma}_{f_1} - \Gamma^k)\succeq 0.
$$
Then, according to (\ref{noneq}), the inequality (\ref{neqsubxzd}) can be rewritten as
\begin{equation}\label{neqsubxzdef}
\begin{array}{l}
\Big\langle z^{k+1} - z^k, A^*(x^k - x^{k+1})\Big\rangle\\[3mm]
\ge \displaystyle{\frac{\rho}{2}}\Big[\|x^{k+1} - x^k\|_{\hat{\Sigma}_{f_1} + \S - \Gamma^k}^2 + \|\widetilde{x}^k - x^k\|_{\widehat{\Sigma}_{f_1} + \S - \Gamma^k}^2 - \|\widetilde{x}^k - x^{k+1}\|_{\hat{\Sigma}_{f_1} + \S - \displaystyle{\frac{1}{2}}\Gamma^k}^2\Big]\\[3mm]
\quad + \|x^{k+1} - x^k\|_{\widehat{\Sigma}_{f_1} + \S}^2\\[3mm]
\ge \displaystyle{\frac{\rho}{2}}\Big[\|x^{k+1} - x^k\|_{\hat{\Sigma}_{f_1} + \S - \Gamma^k}^2 + \|\widetilde{x}^k - x^k\|_{\widehat{\Sigma}_{f_1} + \S - \Gamma^k}^2 - 2\|x^{k+1} - x^{k}\|_{\hat{\Sigma}_{f_1} + \S - \frac{1}{2}\Gamma^k}^2\\[3mm]
\quad - 2\|\widetilde{x}^k - x^{k}\|_{\hat{\Sigma}_{f_1} + \S - \frac{1}{2}\Gamma^k}^2\Big] + \|x^{k+1} - x^k\|_{\widehat{\Sigma}_{f_1} + \S}^2\\[3mm]
= \displaystyle{\frac{2-\rho}{2}}\|x^{k+1} - x^k\|_{\widehat{\Sigma}_{f_1} + \S}^2 - \displaystyle{\frac{\rho}{2}}\|\widetilde{x}^k - x^{k}\|_{\widehat{\Sigma}_{f_1} + \S}^2.
 \end{array}
\end{equation}
Substituting the inequality (\ref{neqsubxzdef}) into (\ref{eqcon}), it is easy to see that our conclusion (\ref{conneq}) holds.
\end{proof}

In light of above lemmas, we are ready to establish the convergence of the G-ADMM-M algorithm under some technical conditions to solve
the problem (\ref{prob1}).

\begin{theorem}\label{convmipgadmm}
Suppose that the Assumption (\ref{assum}) holds. Let the sequence $\{(x^k,y^k,z^k)\}$ be generated by Algorithm G-ADMM-M, and that $(\bar{x},\bar{y},\bar{z})\in \mathbb{W}^*$ is a solution of problem (\ref{prob1}). Let $\Psi_k$, $\delta_k$, $\theta_k$ and $\xi_k$ are defined by (\ref{psik}), (\ref{deltak}), (\ref{thetak}) and (\ref{xik}), respectively. For any $\sigma > 0$, $\rho\in(0,2)$, $\lambda\in(0,1]$, and $k > 1$, it holds that
\begin{equation}\label{convneq}
\begin{array}{l}
\Big(\Psi_k(\bar{x},\bar{y},\bar{z}) + (2-\rho)\|\widetilde{x}^k - x^{k}\|_{\widehat{\Sigma}_{f_1} + \S}^2 +  \sigma(1-\lambda)(2-\rho)\|\A^*x^{k} + \B^*y^{k-1} - c\|^2\Big)\\[3mm]
 - \Big(\Psi_{k+1}(\bar{x},\bar{y},\bar{z}) + (2-\rho)\|\widetilde{x}^{k+1} - x^{k+1}\|_{\widehat{\Sigma}_{f_1} + \S}^2 +  \sigma(1-\lambda)(2-\rho)\|\A^*x^{k+1} + \B^*y^{k} - c\|^2\Big)\\[3mm]
 \ge \delta_k - \xi_k.
 \end{array}
\end{equation}
Moreover, if $\lambda\in(\frac{1}{2},1]$, we also assume that
\begin{equation}\label{pdefneq}
\widehat{\Sigma}_{h_1} + \T \succ 0, \quad \F \succ 0, \quad \H \succ 0,
\end{equation}
\begin{equation}\label{sumxik}
\sum_{k=1}^{\infty}\xi_k < +\infty,
\end{equation}
Then the sequence $\{(x^k,y^k)\}$ converges to an optimal solution of the primal problem (\ref{prob1})) and that $\{z^k\}$ converges to an optimal solution of the corresponding dual problem.
\end{theorem}
\begin{proof}
At the beginning, it holds that
\begin{equation}\label{constieq}
\begin{array}{rl}
&\|\A^*x^{k} + \B^*y^{k} - c\|^2\\[3mm]
= &\|\A^*x^{k} + \B^*y^{k-1} - c + \B^*(y^k - y^{k-1})\|^2\\[3mm]
=&\|\A^*x^{k} + \B^*y^{k-1} - c\|^2 + \|\B^*(y^k - y^{k-1})\|^2 + 2 \Big\langle \A^*x^{k} + \B^*y^{k-1} - c, \B^*(y^k - y^{k-1})\Big\rangle\\[3mm]
\ge& - \|\A^*x^{k} + \B^*y^{k-1} - c\|^2 + \frac{1}{2}\|\B^*(y^k - y^{k-1})\|^2.
 \end{array}
\end{equation}
According to (\ref{noneq}),(\ref{conneq}) and the inequality (\ref{constieq}) above, we have, for any $\sigma > 0$, $\rho\in(0,2)$, $\lambda\in(0,1]$, and $k > 1$, that
\begin{align}\label{thetaconstieq}
&\theta_k + \sigma(2-\rho)\|\A^*x^{k} + \B^*y^{k} - c\|^2\nonumber\\[3mm]
= &\theta_k + (1-\lambda)\sigma(2-\rho)\|\A^*x^{k} + B^*y^{k} - c\|^2 + \lambda\sigma(2-\rho)\|\A^*x^{k} + \B^*y^{k} - c\|^2\nonumber\\[3mm]
\ge& \theta_k + \lambda\Big(\sigma(2-\rho)\|\A^*x^{k+1} + \B^*y^k - c\|^2 + \frac{(2-\rho)^2}{\rho}\|x^{k+1} - x^k\|_{\widehat{\Sigma}_{f_1}+\S+\sigma AA^*}\nonumber\\[3mm]
 &- (2-\rho)\|\widetilde{x}^k - x^k\|_{\widehat{\Sigma}_{f_1}+\S}^2\Big) - (1-\lambda)\sigma(2-\rho)\Big(\|\A^*x^{k} + \B^*y^{k-1} - c\|^2 - \frac{1}{2}\|\B^*(y^k - y^{k-1})\|^2\Big)\nonumber\\[3mm]
=& \|\widetilde{x}^{k+1} - x^{k+1}\|_{\frac{1}{2}\Sigma_{f_1}-\widehat{\Sigma}_{f_1}+(2-\rho)(\widehat{\Sigma}_{f_1}+\S)}^2 + \|\widetilde{y}^{k} - y^{k}\|_{\frac{1}{2}\Sigma_{h_1}-\widehat{\Sigma}_{h_1}+(2-\rho)(\widehat{\Sigma}_{h_1}+\T)}^2\nonumber\\[3mm]
&+ (1-\lambda)\sigma(2-\rho)\Big(\|\A^*x^{k+1} + \B^*y^k - c\|^2 - \|\A^*x^{k} + \B^*y^{k-1} - c\|^2\Big) + (2\lambda-1)\sigma(2-\rho)\|\A^*x^{k+1} + \B^*y^k - c\|^2\nonumber\\[3mm]
&+ \frac{\lambda(2-\rho)^2}{\rho}\|x^{k+1} - x^k\|_{\widehat{\Sigma}_{f_1}+\S+\sigma \A\A^*} - (2-\rho)\|\widetilde{x}^k - x^k\|_{\hat{\Sigma}_{f_1}+\S}^2 +(1-\lambda)(2-\rho)\|\widetilde{x}^k - x^k\|_{\widehat{\Sigma}_{f_1}+\S}^2\nonumber\\[3mm]
& + \frac{\sigma(1-\lambda)(2-\rho)}{2}\|B^*(y^k - y^{k-1})\|^2\nonumber\\[3mm]
= &\|\widetilde{x}^{k+1} - x^{k+1}\|_{(2-\rho)(\widehat{\Sigma}_{f_1}+\S)}^2 - (2-\rho)\|\widetilde{x}^k - x^k\|_{\widehat{\Sigma}_{f_1}+\S}^2 + \sigma(1-\lambda)(2-\rho)\Big(\|\A^*x^{k+1} + \B^*y^k - c\|^2\nonumber\\[3mm]
&- \|\A^*x^{k} + \B^*y^{k-1} - c\|^2\Big) + \|\widetilde{x}^{k+1} - x^{k+1}\|_{\frac{1}{2}\Sigma_{f_1}-\widehat{\Sigma}_{f_1}}^2 + \|\widetilde{y}^{k} - y^{k}\|_{\frac{1}{2}\Sigma_{h_1}-\widehat{\Sigma}_{h_1}+(2-\rho)(\widehat{\Sigma}_{h_1}+\T)}^2\nonumber\\[3mm]
&+ (1-\lambda)(2-\rho)\|\widetilde{x}^k - x^k\|_{\widehat{\Sigma}_{f_1}+\S}^2 + (2\lambda-1)\sigma(2-\rho)\|\A^*x^{k+1} + \B^*y^k - c\|^2\nonumber\\[3mm]
& + \frac{\sigma(1-\lambda)(2-\rho)}{2}\|\B^*(y^k - y^{k-1})\|^2 + \frac{\lambda(2-\rho)^2}{\rho}\|x^{k+1} - x^k\|_{\widehat{\Sigma}_{f_1}+\S+\sigma \A\A^*}.
\end{align}
Using Lemma \ref{lem5}, and substituting (\ref{thetaconstieq}) into (\ref{pivineq}), it can get the inequality (\ref{convneq}) easily.

Noting that $\delta_k \ge 0$ in the case of $\lambda\in(\frac{1}{2},1]$, then, using the conditions (\ref{pdefneq}) and (\ref{sumxik}), we have the following inequality from (\ref{convneq}), that is,
\begin{equation}\label{psikbneq}
\begin{array}{l}
0 \le \Psi_{k+1}(\bar{x},\bar{y},\bar{z}) + (2-\rho)\|\widetilde{x}^{k+1} - x^{k+1}\|_{\widehat{\Sigma}_{f_1} + \S}^2 +  \sigma(1-\lambda)(2-\rho)\|\A^*x^{k+1} + \B^*y^{k} - c\|^2\\[3mm]
\quad \le \Psi_k(\bar{x},\bar{y},\bar{z}) + (2-\rho)\|\widetilde{x}^k - x^{k}\|_{\widehat{\Sigma}_{f_1} + \S}^2 +  \sigma(1-\lambda)(2-\rho)\|\A^*x^{k} + \B^*y^{k-1} - c\|^2 + \xi_k\\[3mm]
\quad \le \Psi_1(\bar{x},\bar{y},\bar{z}) + (2-\rho)\|\widetilde{x}^1 - x^1\|_{\widehat{\Sigma}_{f_1} + \S}^2 +  \sigma(1-\lambda)(2-\rho)\|\A^*x^{1} + \B^*y^{0} - c\|^2 + \sum_{j=1}^k\xi_j\\[3mm]
\quad \le +\infty,
 \end{array}
\end{equation}
which means that the sequences of $\{\Psi_k(\bar{x},\bar{y},\bar{z})\}$,  $\{\|\widetilde{x}^k - x^{k}\|_{\hat{\Sigma}_{f_1} + \S}^2\}$ and $\{\|A^*x^{k+1} + B^*y^{k} - c\|^2\}$ are bounded. According to the definition of $\Psi_k$, we know that the sequences $\{\|z_e^k + \sigma(\rho-1)\A^*x_e^{k}\|^2\}, \{\|\A^*x_e^{k}\|^2\}$,  $\{\|\widetilde{x}_e^{k+1}\|_{\hat{\Sigma}_{f_1}+\S}^2\}$ and $\{\|\widetilde{y}_e^k\|_{\hat{\Sigma}_{h_1}+\T}^2\}$ are also bounded. Moreover, from (\ref{noneq}), we get
\begin{equation}\label{xekxkneq}
\frac{1}{2}\|x_e^{k}\|_{\hat{\Sigma}_{f_1} + \S}^2 \le \|\widetilde{x}^k - x^{k}\|_{\widehat{\Sigma}_{f_1} + \S}^2 + \|\widetilde{x}_e^k\|_{\widehat{\Sigma}_{f_1} + \S}^2.
\end{equation}
Also note that $\{\|A^*x_e^{k}\|^2\}$ is bounded, we have that the sequence of $\{\|x_e^{k}\|_{\hat{\Sigma}_{f_1} + \S+\sigma AA^*}^2\}$ is bounded. Then, the sequence $\{\|x^k\|^2\}$ is bounded because of the condition $\F\succ 0$. Similarly, we can get that the sequence $\{\|z^k\|^2\}$ is also bounded. Now, we prove that the sequence $\{\|y^k\|^2\}$ is bounded, too. For any $k\ge1$, from (\ref{convneq}), we have
\begin{equation}\nonumber
\begin{array}{l}
\sum_{j=1}^k\delta_j \le \Big(\Psi_1(\bar{x},\bar{y},\bar{z}) + (2-\rho)\|\widetilde{x}^1 - x^1\|_{\widehat{\Sigma}_{f_1} + \S}^2 +  \sigma(1-\lambda)(2-\rho)\|\A^*x^{1} + \B^*y^{0} - c\|^2\Big)\\[3mm]
 -\Big(\Psi_{k+1}(\bar{x},\bar{y},\bar{z}) + (2-\rho)\|\widetilde{x}^{k+1} - x^{k+1}\|_{\widehat{\Sigma}_{f_1} + \S}^2 +  \sigma(1-\lambda)(2-\rho)\|\A^*x^{k+1} + \B^*y^{k} - c\|^2\Big) + \sum_{j=1}^k\xi_j\\[3mm]
 \le +\infty.
 \end{array}
\end{equation}
Hence, it gets that $\lim_{k\rightarrow+\infty}\delta_k=0$. By the definition of $\delta_k$, when $k\rightarrow+\infty$, we have
\begin{equation}\label{deltaktend0}
\begin{array}{l}
\|\widetilde{x}^{k+1} - x^{k+1}\|_{\frac{1}{2}\Sigma_{f_1}}^2\rightarrow 0,\quad \|\widetilde{y}^{k} - y^{k}\|_{\frac{1}{2}\Sigma_{h_1}+(2-\rho)(\hat{\Sigma}_{h_1}+\T)}^2\rightarrow 0,\quad\|\widetilde{x}^{k} - x^{k}\|_{\hat{\Sigma}_{f_1}+\S}^2\rightarrow 0,\\[3mm]
\|A^*x^{k+1} + B^*y^{k} - c\|^2\rightarrow 0,\quad \|B^*(y^k-y^{k-1})\|^2\rightarrow 0,\quad \|x^{k+1} - x^k\|_{\hat{\Sigma}_{f_1}+\S+\sigma AA^*}^2\rightarrow 0.\\[3mm]
 \end{array}
\end{equation}
 Then, from $\|\widetilde{y}^{k} - y^{k}\|_{\frac{1}{2}\Sigma_{h_1}+(2-\rho)(\widehat{\Sigma}_{h_1}+\T)}^2\rightarrow 0$ and $\frac{1}{2}\Sigma_{h_1}+(2-\rho)(\widehat{\Sigma}_{h_1}+\T)\succ 0$, we can easily know that $\|\widetilde{y}^{k} - y^{k}\|\rightarrow 0$. So, we know that the sequence $\{\|\widetilde{y}^{k} - y^{k}\|\}$ is bounded, and the sequence $\{\|\widetilde{y}_e^k\|\}$ is bounded too, because of $\hat{\Sigma}_{h_1}+\T\succ 0$. Moreover, from (\ref{noneq}), we have
\begin{equation}\nonumber
\frac{1}{2}\|y_e^{k}\|^2 \le \|\widetilde{y}^k - y^{k}\|^2 + \|\widetilde{y}_e^k\|^2.
\end{equation}
Hence, the sequence $\{\|y^k\|^2\}$ is bounded. Consequently, the sequence $\{(x^k,y^k,z^k)\}$ is bounded, and then there exists at least one subsequence $\{(x^{k_i},y^{k_i},z^{k_i})\}$ such that it converges to a cluster point, say $(x^{*},y^{*},z^{*})$.

In the following, we prove that $(x^{*},y^{*})$ is an optimal solution of the problem (\ref{prob1}) and $z^*$ is the corresponding Lagrange multiplier.
By using the first-order optimality condition of the $x$- and $y$-subproblem in (\ref{mipgadmm}) on the subsequence $\{(x^{k_i},y^{k_i},z^{k_i})\}$, we know that
\begin{equation}\label{fooconki}
\left\{
\begin{array}{l}
0 \in \partial f_2(x^{k_i}) + \F (x^{k_i} - \widetilde{x}^{k_i}) + \nabla f_1(\widetilde{x}^{k_i}) + \sigma \A(\A^*\widetilde{x}^{k_i} + \B^*\widetilde{y}^{k_i} - c) + \A\widetilde{z}^{k_i},\\[3mm]
0 \in \partial h_2(y^{k_i}) + \H (y^{k_i} - \widetilde{y}^{k_i}) + \nabla h_1(\widetilde{y}^{k_i}) + \sigma \B(\A^*x^{k_i} + \B^*\widetilde{y}^{k_i} - c) + Bz^{k_i}.
\end{array}
\right.
\end{equation}
From (\ref{deltaktend0}) and (\ref{fooconki}), we have
\begin{equation}\nonumber
\left\{
\begin{array}{l}
0 \in \partial f_2(x^{*}) + \nabla f_1({x}^{*}) + \A{z}^{*},\\[3mm]
0 \in \partial h_2(y^{*}) + \nabla h_1({y}^{*}) + \B z^{*},\\[3mm]
\A^*x^{*} + \B^*{y}^{*} - c = 0.
\end{array}
\right.
\end{equation}
which implies that $(x^*,y^*)$ is an optimal solution to problem (\ref{prob1}) and $z^*$ is the corresponding Lagrange multiplier.

Finally, we show that $(x^*,y^*,z^*)$ is the unique limit point of the sequence $\{(x^k,y^k,z^k)\}$. Without loss of generality, let $(\bar{x},\bar{y},\bar{z}):=(x^*,y^*,z^*)$, from the definition of $\Psi_k$ and the fact that the subsequence $\{(x^{k_i},y^{k_i},z^{k_i})\}\rightarrow (x^*,y^*,z^*)$ as $i\rightarrow\infty$, we have
$$
\lim_{i\rightarrow+\infty}\Big\{\Psi_{k_i}(\bar{x},\bar{y},\bar{z}) + (2-\rho)\|\widetilde{x}^{k_i} - x^{k_i}\|_{\widehat{\Sigma}_{f_1} + \S}^2 +  \sigma(1-\lambda)(2-\rho)\|\A^*x^{k_i} + \B^*y^{k_i-1} - c\|^2\Big\} = 0.
$$
From $(\ref{psikbneq})$, for any $k>k_i$, we have
\begin{equation}\nonumber
\begin{array}{l}
0 \le \Psi_{k+1}(\bar{x},\bar{y},\bar{z}) + (2-\rho)\|\widetilde{x}^{k+1} - x^{k+1}\|_{\widehat{\Sigma}_{f_1} + \S}^2 +  \sigma(1-\lambda)(2-\rho)\|\A^*x^{k+1} + \B^*y^{k} - c\|^2\\[3mm]
\quad \le \Psi_{k_i}(\bar{x},\bar{y},\bar{z}) + (2-\rho)\|\widetilde{x}^{k_i} - x^{k_i}\|_{\widehat{\Sigma}_{f_1} + \S}^2 +  \sigma(1-\lambda)(2-\rho)\|\A^*x^{k_i} + \B^*y^{k_i-1} - c\|^2 + \sum_{j=k_i}^k\xi_j.\\[3mm]
 \end{array}
\end{equation}
Obviously, combining this inequality with (\ref{deltaktend0}), we get
\begin{equation}\nonumber
\lim_{k\rightarrow+\infty}\Psi_{k}(\bar{x},\bar{y},\bar{z}) = 0,
\end{equation}
which means that
\begin{subequations}
\begin{equation}\label{subzklim}
\lim_{k\rightarrow+\infty}\|z_e^k + \sigma(\rho-1)\A^*x_e^k\|^2 = 0
\end{equation}
\begin{equation}\label{subAkxklim}
\lim_{k\rightarrow+\infty}\|\A^*x_e^k\|^2 = 0
\end{equation}
\begin{equation}\label{subwxklim}
\lim_{k\rightarrow+\infty}\|\widetilde{x}_e^k\|_{\hat{\Sigma}_{f_1} + \S}^2 = 0
\end{equation}
\begin{equation}\label{subwyklim}
\lim_{k\rightarrow+\infty}\|\widetilde{y}_e^k\|_{\hat{\Sigma}_{h_1} + \T}^2 = 0
\end{equation}
\end{subequations}
From (\ref{xekxkneq}), (\ref{deltaktend0}) and (\ref{subwxklim}), we know that $\lim_{k\rightarrow+\infty}\|x_e^k\|_{\hat{\Sigma}_{f_1} + \S}^2 = 0$. Then, from (\ref{subAkxklim}), we get
$$
\lim_{k\rightarrow+\infty}\|x_e^k\|_{\widehat{\Sigma}_{f_1} + \S + \sigma \A\A^*}^2 = 0,
$$
Because of $\widehat{\Sigma}_{f_1} + \S + \sigma \A\A^*\succ 0$, we know that $\lim_{k\rightarrow+\infty}x^k = \bar{x}$. Homoplastically, we get $\lim_{k\rightarrow+\infty}y^k = \bar{y}$.

In a similar way, from (\ref{noneq}), we get
\begin{equation}\nonumber
\frac{1}{2}\|z_e^{k}\|^2 \le \|z_e^k + \sigma(\rho-1)\A^*x_e^k\|^2 + \|\sigma(\rho-1)\A^*x_e^k\|^2,
\end{equation}
then, from (\ref{subzklim}) and $\lim_{k\rightarrow+\infty}x_e^k = 0$, we get $\lim_{k\rightarrow+\infty}z_e^k = 0$, and hence, $\lim_{k\rightarrow+\infty}z^k = \bar{z}$. Then, we know that the sequence $\{(x^k,y^k,z^k)\}$ converges to $(\bar{x},\bar{y},\bar{z})$. The poof is complete.
\end{proof}

\setcounter{equation}{0}
\section{Numerical experiments}\label{section5}
In this section, we evaluate the feasibility and efficiency of the G-ADMM-M algorithm by using some simulated convex composite optimization
problems. All experiments are performed under Microsoft Windows 10 and MATLAB R2020b, and run on a Lenovo Laptop with an Intel Core i7-1165G7 CPU at 2.80 GHz and 8 GB of memory.

In this test, we focus on the following convex composite quadratic optimization problem which has been
considered in \cite{miPADMM} and \cite{MGADMM}, that is
\begin{equation}\label{numexprb1}
\min_{x\in\R^m,y\in\R^n} \ \Big\{\frac{1}{2}\langle y, \Q y \rangle - \langle b, y\rangle + \frac{\chi}{2}\|\Pi_{\R_{+}^m}(\D (d-\H y))\|^2 + \mu\|y\|_1 + \delta_{\R_{+}^m}(x) \ : \ \H y + x = c \Big\}
\end{equation}
where $\Q$ is an $n\times n$ symmetric and positive semidefinite matrix (may not be positive definite); $\chi$ is a nonnegative penalty parameter; $\Pi_{\R_{+}^m}(\cdot)$ denotes the projection onto $\R_{+}^m$;
$\H\in\R^{m \times n}$, $b\in\R^n$, $c\in\R^m$, $\mu > 0$ and $d\le c$ are given data; $\D$ is a positive definite diagonal matrix chosen to normalize each row of $\H$ to have an unit norm.
Observing that problem (\ref{numexprb1}) can be expressed in the from of (\ref{prob1}) if it takes
\begin{equation}\nonumber
\left\{
\begin{array}{l}
h_1(y) = \frac{1}{2}\langle y, \Q y \rangle - \langle b, y\rangle + \frac{\chi}{2}\|\Pi_{\R_{+}^m}(\D (d-\H y))\|^2, \ h_2(y) = \mu\|y\|_1,\\[3mm]
f_1(x) \equiv 0, \ f_2(x) = \delta_{\R_{+}^m}(x), \ \A^* = I, \ \B^* = \H.
\end{array}
\right.
\end{equation}
It is not hard to deduce that the KKT system for problem (\ref{numexprb1}) is the following form:
\begin{equation}\label{numexpkkt}
\nabla h_1(y) + \H^* z + v = 0, \  v\in\partial\mu\|y\|_1, \ \H y+x-c=0, \ x \ge 0, \ z \ge 0, \ x \circ z=0.
\end{equation}
where $z$ and $v$ are the corresponding dual variables, $'\circ'$ denotes the elementwise product.

In the following, for performance comparisons, we also employ the M-GADMM of Qin et al. \cite{MGADMM} with $\rho = 1.9$ and M-ADMM of Li et al. \cite{miPADMM} with the step-length parameters $\tau = 1.618$ to solve the problem (\ref{numexprb1}), respectively.
In our algorithm G-ADMM-M,  we set the parameter $\mu=5\sqrt{n}$ and $d = c-5e$ where $e$ is a vector of all ones. Besides, we choose the zero vector as an initial point, and the penalty parameter is initialized as $\sigma = 0.8$. Based on the KKT system (\ref{numexpkkt}), we stop the iterative process based on the residual of the KKT system, that is,
\begin{equation}\nonumber
Res:= \max\Big\{\frac{\|\H y^{k+1} + x^{k+1} -c\|}{1+\|c\|}, \frac{\|\nabla h_1(y^{k+1}) + \H^*z^{k+1} + v^{k+1}\|}{1+\|b\|}\Big\}\le 10^{-5}.
\end{equation}
In the numerical experiments, we choose the same proximal terms and operators as in \cite{miPADMM} and \cite{MGADMM}, i.e.
\begin{equation}\nonumber
\begin{array}{l}
\widehat{\Sigma}_{h_1} = \Q + \chi \H^*\D^2\H, \ \Sigma_{h_1}=\Q, \  \widehat{\Sigma}_{f_1} = \Sigma_{f_1}= 0,\\[3mm]
\T = \lambda_{\max}(\widehat{\Sigma}_{h_1} + \sigma \H^*\H )\I - (\widehat{\Sigma}_{h_1} + \sigma \H^*\H), \ \S = 0.
\end{array}
\end{equation}
With these choices, it is easy to verify that the conditions (\ref{pdefneq}) in Theorem (\ref{convmipgadmm}) are satisfied. Using these notations, we choose different pairs of $(m, n)$ to generate the data used in (\ref{numexprb1}) and the numerical results derived by each algorithm are listed in Tables (\ref{tab1} - \ref{tab2}), which includes the number of iterations (Iter), the
computing time (Time), the KKT residuals of the solution (Res). From these tables, we find that the G-ADMM-M is beter than M-GADMM and M-ADMM in terms of iteration number and computation time almost
when a similar accuracy solutions are achieved. In particular, we can see that our proposed algorithm is $15-35\%$ faster than the M-ADMM.
\begin{table}[ht]
\centering
\renewcommand\tabcolsep{4.5pt}
 {\scriptsize
 \caption{\small Numerical results of M-ADMM, M-GADMM and G-ADMM-M with $\chi=0$.}\vspace{-.1cm}
\begin{tabular}{lc|ccc|ccc|ccc}
\hline
\multirow{2}{*}{m}&
\multirow{2}{*}{n}&
 \multicolumn{3}{c|}{M-ADMM} & \multicolumn{3}{c|}{M-GADMM} & \multicolumn{3}{c}{G-ADMM-M} \\
 \cline{3-11}
    & & Iter & Time & Res & Iter & Time & Res & Iter & Time & Res \\
\hline
200&	500&	5966&	1.44&	9.98e-06&	5032&	1.10&	9.98e-06&	4886&	1.16&	9.98e-06 	\\
500&	200&	576&	0.18&	1.00e-05&	581&	0.17&	9.99e-06&	423&	0.14&	9.96e-06 	\\
500&	500&	794&	1.83&	9.96e-06&	747&	1.59&	9.96e-06&	628&	1.33&	1.00e-05 	\\
1000&	500&	628&	2.80&	9.94e-06&	579&	2.50&	9.98e-06&	444&	1.88&	9.99e-06 	\\
500&	1000&	1354&	4.98&	9.97e-06&	1179&	4.27&	9.90e-06&	1121&	4.12&	9.97e-06 	\\
1000&	1000&	758&	6.39&	9.89e-06&	696&	5.72&	9.98e-06&	558&	4.64&	9.92e-06 	\\
2000&	1000&	630&	11.11&	9.90e-06&	608&	10.61&	9.93e-06&	439&	7.73&	9.96e-06 	\\
1000&	2000&	3088&	46.15&	1.00e-05&	2573&	38.23&	1.00e-05&	2667&	40.05&	1.00e-05 	\\
2000&	2000&	778&	24.39&	9.88e-06&	720&	22.18&	9.96e-06&	550&	17.31&	1.00e-05 	\\
4000&	2000&	587&	41.89&	9.95e-06&	575&	39.78&	9.97e-06&	397&	28.32&	9.88e-06 	\\
2000&	4000&	1816&	105.28&	1.00e-05&	1553&	89.32&	1.00e-05&	1531&	90.12&	9.99e-06 	\\
4000&	4000&	708&	89.75&	9.97e-06&	683&	86.87&	9.99e-06&	524&	68.35&	9.99e-06	\\
4000&	8000&	1334&	320.39&	9.99e-06&	1103&	262.73&	9.99e-06&	1034&	254.14&	9.99e-06 	\\
8000&	4000&	681&	194.44&	9.95e-06&	708&	197.55&	9.99e-06&	487&	141.90&	9.95e-06 	\\
8000&	8000&	793&	460.87&	9.98e-06&	779&	415.07&	9.97e-06&	547&	306.90&	9.99e-06 	\\
\hline
\end{tabular}\label{tab1}
}
\end{table}

\begin{table}[ht]
\centering
\renewcommand\tabcolsep{4.5pt}
 {\scriptsize
 \caption{\small Numerical results of M-ADMM, M-GADMM and G-ADMM-M with $\chi=2\mu$.}\vspace{-.1cm}
\begin{tabular}{lc|ccc|ccc|ccc}
\hline
\multirow{2}{*}{m}&
\multirow{2}{*}{n}&
 \multicolumn{3}{c|}{M-ADMM} & \multicolumn{3}{c|}{M-GADMM} & \multicolumn{3}{c}{G-ADMM-M} \\
 \cline{3-11}
    & & Iter & Time & Res & Iter & Time & Res & Iter & Time & Res   \\
\hline
200&	500&	6221&	1.44&	9.99e-06&	5221&	1.18&	9.99e-06&	5004&	1.28&	1.00e-05 	\\
500&	200&	593&	0.19&	9.92e-06&	568&	0.21&	9.89e-06&	385&	0.13&	9.31e-06 	\\
500&	500&	526&	1.08&	9.87e-06&	495&	0.97&	9.86e-06&	360&	0.88&	9.58e-06 	\\
1000&	500&	458&	2.09&	1.00e-05&	445&	1.95&	9.93e-06&	296&	1.25&	9.96e-06 	\\
500&	1000&	1715&	6.49&	1.00e-05&	1423&	5.21&	9.99e-06&	1519&	5.53&	9.99e-06 	\\
1000&	1000&	430&	3.72&	9.96e-06&	409&	3.46&	9.94e-06&	303&	2.51&	9.65e-06 	\\
2000&	1000&	445&	7.86&	9.86e-06&	441&	7.76&	9.91e-06&	279&	4.95&	9.90e-06 	\\
1000&	2000&	1466&	21.58&	1.00e-05&	1249&	18.57&	9.97e-06&	1296&	19.57&	9.99e-06 	\\
2000&	2000&	341&	10.87&	9.88e-06&	325&	10.39&	9.86e-06&	237&	7.73&	9.79e-06 	\\
4000&	2000&	433&	30.80&	9.97e-06&	430&	30.06&	9.97e-06&	263&	19.15&	9.82e-06 	\\
2000&	4000&	1276&	74.11&	9.99e-06&	1075&	62.47&	9.96e-06&	1104&	65.02&	9.96e-06 	\\
4000&	4000&	312&	42.04&	9.96e-06&	306&	41.44&	9.94e-06&	211&	29.88&	9.99e-06	\\
4000&	8000&	1028&	253.05&	1.00e-05&	841&	203.23&	9.97e-06&	903&	224.51&	9.99e-06 	\\
8000&	4000&	463&	135.72&	9.98e-06&	470&	134.67&	9.92e-06&	298&	90.87&	9.92e-06 	\\
8000&	8000&	315&	188.71&	9.91e-06&	314&	183.85&	9.99e-06&	207&	136.79&	9.85e-06 	\\
\hline
\end{tabular}\label{tab2}
}
\end{table}

\section{Concluding remarks}\label{section6}

We have known that the generalized ADMM of Xiao et al. \cite{spGADMM} has the ability to solve two-block linearly constrained composite convex programming problem, in which the smooth term located in each block must be `quadratic'.
This paper enhanced the capacity of this method and showed that if we used a majorized quadratic function to replace the non-quadratic smooth function, the resulting subproblem in the ADMM framework can be split some smaller problems if a symmetric Gauss-Seidel iteration is used.
The majority of this paper is on the theoretical analysis,  that is, we proved that the sequence generated by the proposed method converges to one of a KKT points of the considered problem. Besides, we also did some simulated experiments to illustrate the effectiveness of the majorization technique and the progressiveness of the proposed method. At last, we must note that, further performance evaluation of the method using various problems deserves investigating. Of course, research on the proposed method's convergence rate and its iteration complexity are also interesting topics for further research.

\section*{Acknowledgments}
We would like to thank Z. Ma from Henan University for her valuable comments on the proof of this paper. The work of Y. Xiao is supported by the National Natural Science Foundation
of China (Grant No. 11971149).


\begin{thebibliography}{30}


\bibitem{CHEND} C. H. Chen, \emph{Numerical algorithms for a class of matrix norm approximation problems}, Nanjing University, 2012.

\bibitem{noteADMM} L. Chen, D. F. Sun, and K.-C. Toh, \emph{A note on the convergence of ADMM for linearly constrained convex optimization Problems}, Comput. Optim.Appl. 2017, 66(2), 327-343.

\bibitem{OPNSA} F.H. Clarke, \emph{Optimization and Nonsmooth Analysis}, SIAM, Philadelphia, 1990, vol.5.

\bibitem{CUIJOTA} Y. Cui, X. D. Li, D. F. Sun, and K.-C.Toh,\emph{On the convergence properties of a majorized ADMM for linearly constrained convex optimization problems with coupled objective functions}. J. Optim.
Theory Appl. 2016, 169(3), 1013-1041.


\bibitem{GADMM} J. Eckstein  and D. P. Bertsekas, \emph{On the Douglas-Rachford splitting method and the proximal point algorithm for maximal monotone operators}, Math. Program., 1992, 55, 293-318.

\bibitem{reEY} J. Eckstein and W. Yao, \emph{Understanding the convergence of the alternating direction method of multipliers: theoretical and computational perspectives}, Pac. J. Optim. 2014, 11(4), 619-644.

\bibitem{FST}M. Fazel, T. K. Pong, D. F. Sun, and P. Tseng, \emph{Hankel matrix rank minimization with applications to system identification and realization}, SIAM J. Matrix Anal. Appl., 2013, 34, 946-977.

\bibitem{FGC} M. Fortin and R. Glowinski, \emph{Augmented Lagrangian methods: applications to the numerical solution of boundary-value problems}, Studies in mathematics and its applications, Elsevier Science Publishers B.V.,1983, 97-146.

\bibitem{ADMMDRS} D. Gabay, \emph{Applications of the method of multipliers to variational inequalities in Augmented Lagrangian Methods: Applications to the Numerical Solution of Boundary-Value Problems}, In: Fortin, M. and Glowinski, R. (eds.): Studies in Mathematics and Its Applications, 1983, 15, 299-331.

\bibitem{GMC} D. Gabay and B. Mercier, \emph{A dual algorithm for the solution of nonlinear variational problems via finite element approximation}, Comput. Math. Appl., 1976, 2, 17-40.


\bibitem{GHC} R. Glowinski,  \emph{On alternating direction methods of multipliers: A historical perspective}, Modeling, Simulation and Optimization for Science and Technology, Springer 2014, 59-82.

\bibitem{HONG} M. Hong, T. H. Chang, X. Wang, et al., \emph{A block successive upper bound minimization method of multipliers for linearly constrained convex optimization}. arXiv: 1401.7079, 2014.

\bibitem{miPADMM} M. Li , D. F.Sun, and K.-C.Toh, \emph{A majorized ADMM with indefinite proximal terms for linearly constrained convex composite optimization}, SIAM J. Optim. 2016, 26, 922-950.

\bibitem{spADMM} X. D. Li,  D. F. Sun, and K.-C.Toh, \emph{A Schur complement based semi-proximal ADMM for convex quadratic conic programming and extensions}, Math. Program. 2016, 155, 333-373.

\bibitem{bsGs} X. D. Li,  D. F.Sun, and K.-C.Toh, \emph{A block symmtric Gauss-Seidel decomposition theorem for convex composite quadratic programming and its applications}, Math. Program. 2019, 175, 395-418.


\bibitem{MGADMM} C. Qin, Y. Xiao, and P. Li, \emph{A majorized-generalized alternating direction method of multipliers for convex composite programming},  arXiv:2111.12519, 2021.

\bibitem{COVA} R. T. Rockafellar, \emph{Convex analysis}, Princeton University Press, 1970.

\bibitem{ALMPPA} R. T. Rockafellar, \emph{Augmented Lagrangians and applications of the proximal point algorithm in convex programming}, Math. Oper. Res., 1976, 1(2), 97-116.


\bibitem{spGADMM} Y. Xiao, L. Chen, and D. H. Li, \emph{A generalized alternating direction method of multipliers with semiproximal terms for convex composite conic programming}, Math. Program. Comput., 2018, 10(4), 533-555.



\end{thebibliography}
\end{document}